\documentclass{article}
\usepackage[utf8]{inputenc}
\usepackage{amsmath,amssymb,amsthm,ulem}
\usepackage{hyperref}
\usepackage{authblk}
\hypersetup{
	colorlinks=true,
	linkcolor=red,
	filecolor=magenta,      
	urlcolor=cyan,
	pdftitle={Overleaf Example},
	pdfpagemode=FullScreen,
}
\usepackage{url}
\usepackage{geometry}
\usepackage{graphicx}
\usepackage{natbib}
\usepackage[dvipsnames,usenames]{xcolor} 
\usepackage{mathtools}
\usepackage{cleveref}
\usepackage{tikz}
\usepackage{comment}
\usepackage[linesnumbered,ruled,vlined]{algorithm2e}
\usepackage{enumitem}

\usepackage{multirow}
\usepackage{caption}
\usepackage{subcaption}
\newtheorem{theorem}{Theorem}
\numberwithin{theorem}{section}
\newtheorem{proposition}[theorem]{Proposition}
\newtheorem{lemma}[theorem]{Lemma}

\theoremstyle{definition}
\newtheorem{definition}[theorem]{Definition}

\newtheorem{remark}[theorem]{Remark}
\newtheorem{example}[theorem]{Example}
\newtheorem{conjecture}[theorem]{Conjecture}

\newtheoremstyle{example_contd}
{0.3cm}
{0.3cm}
{\upshape}
{}
{\bfseries\scshape}
{.}
{0.5em}
{ \thmname{#1} \thmnumber{ #2}\thmnote{#3} (continued)}

\theoremstyle{example_contd}
\newtheorem*{example_contd}{\hspace*{-0.45em}Example}

\graphicspath{{Figures/}}

\newcommand{\vor}{\mathrm{Vor}}

\newcommand{\N}{\mathcal{N}}
\newcommand{\F}{\mathcal{F}}
\newcommand{\X}{\mathcal{X}}

\newcommand{\RR}{\mathbb{R}}
\newcommand{\Rd}{\mathbb{R}^{n*}}
\newcommand{\PP}{{\mathbb{P}}}
\newcommand{\Rp}{{\mathbb{RP}}}

\newcommand{\type}{\mathrm{Type}}

\newcommand{\Conv}{\mathrm{Conv}}
\newcommand{\med}{\mathrm{Med}_h}
\newcommand{\IF}{I_{(F_1,F_2)}}
\newcommand{\JF}{J_{(F_1,F_2)}}
\renewcommand{\l}{\lambda}
\newcommand{\V}{{\mathbb{V}}}
\newcommand{\pn}{{\psi(\N(X))}}

\renewcommand{\L}{\mathcal{L}}
\renewcommand{\S}{\mathcal{S}}

\DeclareMathOperator{\res}{{Res}}
\DeclareMathOperator{\vc}{{VCone}}

\title{Algebraic Distance Optimization in Polyhedral Norms}

 \author{Eliana Duarte, Nidhi Kaihnsa, Julia Lindberg,\\ Angélica Torres, and Madeleine Weinstein}

\begin{document}
	\maketitle
    \begin{abstract}
        We consider the distance minimization problem to a real algebraic variety $X \subseteq \RR^n$ when
        the metric is induced by a polyhedral norm. Each point in the variety has a Voronoi cell whose geometry
        depends on the normal space at the point and the inner normal fan of the polyhedral ball. 
        For codimension-one varieties, we decompose $X$ into sets of  points whose Voronoi cones have the same dimension, which is the expected dimension of their Voronoi cell. 
        We prove that this decomposition is a stratification of $X$ and 
        that each strata is a semialgebraic set.
        We conclude by giving an algebraic description of the medial axis, which is the locus of points whose minimal distance to $X$ is achieved at more than one point on $X$.
    
     \end{abstract}
     
     \section{Introduction}
A distance optimization problem consists of a subset $X$ of a metric space, representing a model, together with an observed point, $v \not\in X$.
     The fundamental question is then: which point of $X$ best explains the observed point $v$? Algebraic models of this type arise in disciplines such as algebraic statistics, computer vision, machine learning, and phylogenetics, among others \cite{breiding2024metric}. The complexity of this optimization problem has been studied from an algebraic perspective for a wide range of models $X$ and distance functions \cite{Breiding2022,breiding2024metric,Draisma2016,Maxim2020}. 

     In this article, we consider the inverse problem. Given a set $X\subset \RR^n$ and a point $p\in X,$ we ask which points in $\RR^n$ are best modeled by $p$. 
     In other words, we study which data points have their distance to $X$ minimized at $p$.
     The set of all such points in $\RR^n$ is called the \textit{Voronoi cell} of $p$. The collection of all of the Voronoi cells forms the \textit{Voronoi diagram} of $X$.

     Voronoi diagrams have been studied extensively, predominantly in the setting where $X$ is a finite subset of $\RR^n$ equipped with the Euclidean metric. In \cite{CIFUENTES2022351}, the authors extend this analysis to the case when $X$ is an algebraic variety, showing that for any smooth point $p\in X$, the Euclidean Voronoi cell is convex and contained in the normal space of $X$ at $p.$ 
        For discrete sets and the tropical norm, the authors in \cite[Theorem 6]{Criado} proved that Voronoi cells are star-convex, but not necessarily convex. 
     In this work, we study the Voronoi diagram when $X$ is an algebraic variety and the metric is induced by a polyhedral norm. 

     A primary motivation for this setting comes from optimal transport. Consider the probability simplex of all distributions on $n$ states. A metric $d$ on the set of states induces a transportation distance on the simplex known as the Wasserstein distance $W_d(\mu,\nu)$. The corresponding unit ball is a centrally symmetric convex polytope known as the \textit{Lipschitz polytope}. This setting was studied in \cite{CELIK_Optimal_transport,CELIK_Wasserstein}, where it was shown that the Wasserstein distance to a variety is a piecewise polynomial function.

     Our approach follows \cite{BECEDAS2024102229, CIFUENTES2022351}. For a point $p\in X,$ \cite{BECEDAS2024102229} shows that its Voronoi cell is contained in a union of polyhedral cones, which we refer to as the \textit{Voronoi cone} of $p.$ We prove that the Voronoi cone is determined by the relative position of the normal space of $X$ at $p$ with respect to the inner normal fan of the unit ball. Equivalently, the Voronoi cone can be characterized by the intersection of the dual unit ball with the normal space at $p$ (\Cref{thm:generalresult}). Moreover, when $X$ has pure codimension one, this perspective allows us to stratify $X$ according to the dimension of the Voronoi cones of its points (\Cref{thm:stratgeneral}). 

     We also study the \textit{medial axis} of $X,$ the locus of points whose distance to $X$ is optimized by more than one point. 
     When $X$ consists of finitely many points, prior work examined the combinatorial structure of such bisectors \cite{Criado,JalJochemko}. We consider the case where $X$ has codimension one.
     In this case, the medial axis decomposes into components associated with pairs of faces of the unit ball. In the extreme cases (pairs consisting of two vertices, two facets, or a vertex and a facet) we obtain degree bounds for the algebraic closure of the corresponding component of the medial axis (\Cref{thm:vertex_vertex,thm:vertex-facet,thm:facet-facet}).
     
     The paper is organized as follows. In \Cref{sec:generalstratify}, we review the necessary background from polyhedral geometry, polyhedral norms, and differential topology. In \Cref{sec:VDbasic}, we focus on the relationship between the normal space of a point in $X$, the normal fan of the unit ball, and the Voronoi cone of $p$. In particular, we show that the Voronoi cell of a point depends on the geometry of $X$ with respect to the norm. In \Cref{sec:Codim1}, we show that for codimension-one varieties we can stratify $X$ based on the dimension of the Voronoi cones of points on $X$ and give a recipe to determine this stratification. Finally, in \Cref{sec:MA}, we give an algebraic description of the medial axis and establish degree bounds for its components.

\section{Preliminaries}\label{sec:generalstratify}

\subsection{Polyhedral Geometry}

To study the Voronoi diagram of a real algebraic variety $X$ associated to a polyhedral norm, we recall several definitions related to the geometry of convex polytopes, their duals, and the polyhedral cones associated to their faces. We define only the objects that will be used in our study and refer the reader to \cite{deloera2010,ziegler2012} for a detailed exposition. In what follows, we will denote by $\RR^{n*}$ the dual space of $\RR^n.$

\begin{definition}
	Let $P$ be a convex polytope in $\RR^n$ such that the origin is in its interior. 
	The \emph{dual polytope} to $P$, denoted by $P^*\subset \RR^{n*}$, is defined as
	\begin{align*}
		P^*:=\{y\in\RR^{n*}~|~x^Ty\leq 1~ \forall ~x \in P\}.
	\end{align*}
\end{definition}

Throughout this article, we  consider full-dimensional convex polytopes in the ambient space. 

	\begin{definition}
		Let $\{v_1,\ldots, v_k\} \subseteq \RR^n$ be a set of nonzero vectors. 
		The cones
		\begin{align*}
			C:=\left\{\sum_{i=1}^k\lambda_iv_i \mid (\lambda_1,\ldots,\lambda_k) \in \RR_{> 0}^k \right\} \qquad \text{ and } \qquad \overline{C}:=\left\{\sum_{i=1}^k\lambda_iv_i \mid (\lambda_1,\ldots,\lambda_k)\in\RR_{\geq 0}^k\setminus \{0\}\right\}
				\end{align*}
		are, respectively, the {\em open and closed} polyhedral cones generated by $\{v_1,\ldots, v_k\}$ in $\mathbb{R}^n \backslash \{0\}$.
	\end{definition}

	In the Euclidean topology, $\overline{C}$ is the closure of $C$ in $\RR^n\setminus \{0\}$. 
    We will work primarily with open cones. Thus, unless stated otherwise, ``cone" will refer to an open cone.
    We denote the set of (proper) faces of a polytope as $\F_P\coloneqq \{F~|~ F \text{ is a proper face of }P\}$. Note that $\emptyset, P \notin \F_P$. Furthermore, we define $\F_{P,i}\coloneqq \{F~|~ F \in \F_P \text{ and } \dim F= n-1-i\}$ for $i=0,\ldots,n-1$. When the polytope is clear from the context, we write simply $\F$ (resp. $\F_{i}$) instead of $\F_P$ (resp. $\F_{P,i}$). 
	A hyperplane $H$ is a supporting hyperplane of a polytope $P$ at $F$ if 
	$H\cap P=F$ for some $F\in \mathcal{F}$ and $P\subset H^{+}$ where $H^+\coloneqq \{x \in\RR^n \mid H(x)\geq 0 \}.$ 
	Given a polytope, we associate to it a polyhedral fan via its faces.	
	\begin{definition}
		Let $P$ be a full-dimensional convex polytope in $\RR^n$ and $F$ one of its faces. The \emph{inner normal (open) cone}, $C(P,F)$, at $F$ is the open cone containing all of the inner pointing normal vectors of all of the supporting hyperplanes $H$ of $P$ 
			such that $P\cap H= F$. In particular, this cone is generated by the normal vectors of the supporting hyperplanes of all of the facets containing $F.$
	\end{definition}

		Whenever the polytope $P$ is clear from context, we write simply $C_F$ for the cone $C(P,F)$ and denote ${\rm gens}(C_F)$ as the minimal generators of this cone. Additionally, for $v\in\RR^n$, we denote the translate of $C(P,F)$ to $v$ by $$C_v(P,F):=v+C(P,F).$$
	
	\begin{remark}\label{rem:coneclosure}
		From the above definitions, it is easy to check that for any $F\in \F$, $\overline{C_F}=\bigsqcup_{F\subseteq F'}C_{F'}.$
	\end{remark}
	
	The union of the inner normal cones of all of the faces of a polytope $P$ is called the \emph{inner normal fan} of $P$; it is denoted by $\Sigma_P$ and it partitions $\RR^n\setminus \{0\}$ into polyhedral regions indexed by the faces of $P$.

	\begin{remark}\label{rem:normalcone}
		Consider a polytope $P$ in $\RR^n$ and its dual $P^*$. Let $F$ be a face of $P$ and $v$ be a vector in $C(P,F)$. We can consider $v$ as a linear functional over $\Rd$. 
        From the definition of duality, it is straightforward to check that $v$ attains its minimum over $P^*$ precisely on the face $F^* \subseteq P^*$ dual to $F$.
	\end{remark}

		\begin{example}\label{ex:plane_curve}
			Let $b_1:=(-1,1)$, $b_2:=(1,1)$, $b_3:=(1,-1)$, and $b_4:=(-1,-1)$. Consider the polytope $P= \Conv\{b_1,b_2,b_3,b_4\}.$ This polytope is a square in $\RR^2$ and it is depicted in~\Cref{fig:square-diamond-fans}.  The four vertices of $P$ have two-dimensional inner normal cones. These are exactly the four blue orthants of the plane excluding the points on the corresponding axes. The four one-dimensional cones, corresponding to the four edges, are the rays starting at the origin in the direction of the coordinate axes.  
			The dual polytope $P^*=\Conv\{(0,1),(0,-1),(1,0),(0,-1)\}$ is a diamond. The one-dimensional cones
            of $P^*$
            associated to the facets are the rays starting in the origin and in the direction of the vectors $\{\pm(1,1), \pm(-1,1)\}$. The two-dimensional cones associated to the vertices are depicted in orange in~\Cref{fig:square-diamond-fans}.

			\begin{figure}[h]
            \centering
	\begin{subfigure}[t]{0.2\textwidth}
		\centering 
		\includegraphics[scale=0.3]{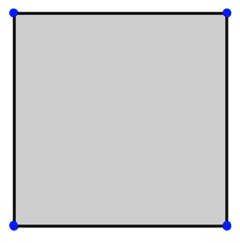}
	\end{subfigure}
 \begin{subfigure}[t]{0.2\textwidth}
		\centering 
		\includegraphics[scale=0.3]{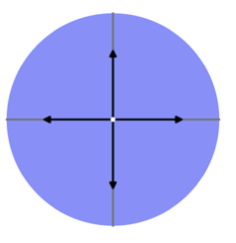}
	\end{subfigure}
    \hspace{0.6cm}
	\begin{subfigure}[t]{0.2\textwidth}
		\centering
		\includegraphics[scale=0.3]{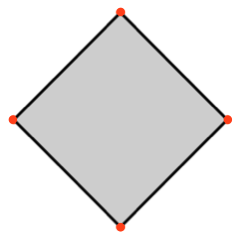}
	\end{subfigure}
     \begin{subfigure}[t]{0.2\textwidth}
		\centering
		\includegraphics[scale=0.3]{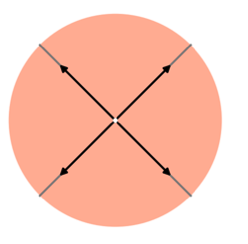}
	\end{subfigure}
				\caption{From left to right, we have $P$ and its inner normal fan, and $P^*$ and its inner normal fan.}
				\label{fig:square-diamond-fans}
			\end{figure}
		\end{example}

		\begin{example}\label{ex:hypersurface}
			Consider the three dimensional cube 
			\[
			P= \Conv\{(1,1,1),(1,-1,1),(-1,-1,1),(-1,1,1),(1,1,-1),\\(1,-1,-1),(-1,-1,-1),(-1,1,-1)\}.
			\]
            The inner normal fan consists of $8$ three-dimensional cones (one for each vertex), which coincide with the open orthants; $12$ two-dimensional cones (one for each edge), lying in the coordinate planes; and $6$ one-dimensional cones (one for each facet), given by the rays from the origin along the coordinate axes.
            The dual ball $P^*$ is the octahedron. The polyhedral ball $P$, its dual $P^*$, and its associated inner normal fan are depicted in \Cref{fig:cube-octa-fans}.
			\begin{figure}[h]
				\centering
				\includegraphics[scale=0.5]{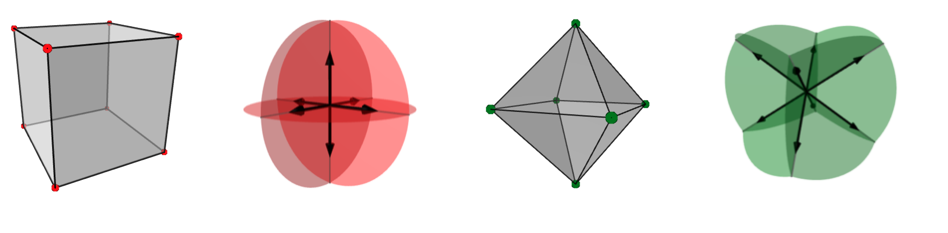}
				\caption{From left to right, we have $P$, the inner normal fan of $P$, $P^*$, and the inner normal fan of $P^*$.}
				\label{fig:cube-octa-fans}
			\end{figure}

		\end{example}

\subsection{Polyhedral Norms}

Let $\L$ be a finite set of linear functionals on $\RR^n$ such that the set $P = \{x\in \RR^n~|~ \ell(x)\leq 1 \text{ for } \ell\in \mathcal{L}\}$ is bounded and describes a full-dimensional convex polytope. Furthermore, we impose that $\mathcal{L}=-\mathcal{L}$
so the polytope $P$ is centrally symmetric around the origin, i.e. $P=-P$. For the entirety of this article, we will assume that $\mathcal{L}=-\mathcal{L}$ and $P=-P.$

	\begin{definition}
		For a finite collection of linear functionals $\L$, a \emph{polyhedral norm} is a function $h_{\L}$ on $\RR^n$ given by 
        $$h_{\L}(x):=\max_{\ell \in \mathcal{L}} \ \ell(x).$$
	\end{definition}

It is a straightforward exercise to verify that polyhedral norms satisfy the norm axioms.
Under this norm, the distance $d(x,y)$ between two points $x,y\in\RR^n$ is $d(x,y)\coloneqq h_\L(x-y)$.
Furthermore, given a set $X\subset \RR^n$, we define the \emph{polyhedral distance} of a point $y\in \RR^n$ to $X$ as 
\begin{equation*}
d(y,X):=\min_{x\in \overline{X}} \ h_\L(y-x),
\end{equation*}
where $\overline{X}$ is the Euclidean closure of $X \subseteq \RR^n$.

\begin{definition}
	Given a polyhedral norm $h_{\L}$, the \textit{ball} of radius $r$ at an arbitrary point $y \in \RR^n$ is \[B_r(y):=\{x\in \RR^n~|~h_\L(x-y)\leq r\}.\] For convenience, we write the unit ball at $y$ as $B(y)$ and simply as $B$ when considering the unit ball around the origin. 
\end{definition}

For a given polyhedral norm, $B$ is a centrally symmetric convex polytope. 
In the next result, we show that there is a polyhedral norm associated with every full-dimensional centrally symmetric convex polytope.

\begin{lemma}
	Let $P$ be a full-dimensional centrally symmetric convex  polytope in $\RR^n$ that contains the origin. Then there is a polyhedral norm $h_\L$ such that $P$ is the unit ball for $h_\L.$
\end{lemma}

\begin{proof}
    We consider the set of linear functionals $\L$ that give supporting hyperplanes for the facets of $P$. This gives the minimal $\mathcal{H}-$representation of $P$. Since $0\in P$ and $P$ is centrally symmetric, the set $\L$ is of the form $\{a_1,\ldots, a_m,-a_1,\ldots,-a_m\}$ with $a_i\in\RR^{n*}$. By construction, $P$ is the unit ball of the polyhedral norm $h_\L$ up to scaling. 
\end{proof}

\begin{example_contd}[\ref{ex:plane_curve}]
	The square is the unit ball for the polyhedral norm defined by the set of functionals $\left\{\pm x, \pm y\right\}.$
	The dual polytope $P^{\ast}$ is the unit ball for the dual norm and it is defined by the set of functionals
	$ \left\{ \pm(x+y), \pm( -x+y)\right\}.$
\end{example_contd}

\begin{example_contd}[\ref{ex:hypersurface}]
	Similar to the previous example, the cube is the unit ball for the norm defined by the set of functionals
	$\left\{\pm x, \pm y,\pm z\right\}.$
	The unit ball for the dual norm is the octahedron and
	it is defined by the functionals 
	$
	\left\{\pm(x+y+z), \pm( x-y+z), \pm(x+y-z),\pm(x-y-z) \right\}.
$
\end{example_contd}

\subsection{Differential Topology} \label{sec:top}

Here, we recall some results and definitions from differential topology that we will use later to stratify a codimension-one manifold using polyhedral norms. For a basic introduction to topological and smooth manifolds, we refer readers to \cite{lee2,lee1}. 
Going forward, we consider a manifold $X \subset \RR^n$ with the Euclidean topology inherited from $\RR^n$.  Below, we give definitions of locally closed subspaces and stratifications of $X$.

\begin{definition}
	A subset $U$ of a topological space $X$ is \textit{locally closed} if $\overline{U}\setminus U$ is closed in $X$. Equivalently, $U$ is locally closed if it is the set difference of two open sets in $X$.
\end{definition}

\begin{definition}\label{def:Strat_top_space}
    Let $X$ be a manifold and $I$ be an index set with a partial ordering. The collection $\{X_i\}_{i\in I}$ is a \textit{stratification} of $X$ if the following conditions are satisfied:
    \begin{enumerate}[label=(\roman*)]
        \item \label{def212:it1} $X=\bigcup_{i\in I}X_i$,
        \item \label{def212:it2}  $X_i$ is a submanifold for all $i\in I$,
        \item \label{def212:it3}  $X_i$ is locally closed for all $i\in I$,
        \item \label{def212:it4}  $X_i\cap X_j =\emptyset $ if $i\neq j$,
        \item \label{def212:it5}  if $X_i\cap \overline{X_j}\neq\emptyset$, then $X_i\subset \overline{X_j}$.  Equivalently, if $X_i\subset \overline{X_j}$, then $i\leq j$.
    \end{enumerate}
\end{definition}

			A family $\{X_i\}_{i\in I}$ of subsets of a topological space $X$ is \textit{locally finite} if, for each $x\in X$, there is a neighborhood $U$ of $x$ such that $U\cap X_i=\emptyset$ for all but finitely many indices $i\in I.$ If $I$ is a finite set, that is, $\{X_i\}_{i\in I}$ is a finite family, then it is trivially locally finite. This will always be the case later on.

We now recall some basic results and definitions that will be useful when constructing the stratification. 

\begin{lemma}[\cite{bourbaki} Proposition 2, Ch. 5.1]\label{lem:restr_proj}
	Let $f: X\longrightarrow Y$ be a map between two topological spaces $X$ and $Y,$ and fix a subset $T\subset Y$. If $f$ is an open map (resp. closed, continuous), then so is the restriction $f|_{f^{-1}(T)}:f^{-1}(T)\longrightarrow T$. 
\end{lemma}

Let $X$ and $Y$ be smooth manifolds of dimensions $m$ and $n$, respectively. A map $f:X\longrightarrow Y$ is an \textit{immersion at a point} $x\in X$ if $df_x:T_xX\longrightarrow T_{f(x)}Y$ is injective. If $f$ is an immersion at $x$ for all $x\in X$, then $f$ is called an \textit{immersion}. Furthermore, $f$ is a \textit{smooth embedding} of $X$ into $Y$ if it is a smooth immersion and a topological embedding, i.e., a homeomorphism onto its image $f(X)\subseteq Y$ in the subspace topology. A subset $S$ is an \textit{embedded submanifold} of $X$ in the subspace topology if it is endowed with a smooth structure with respect to which $S\hookrightarrow X$ is a smooth embedding.

Let $f:X\to Y$ be a smooth map and $S\subseteq Y$ be an embedded submanifold of $Y$. We say that $f$ is \textit{transverse} to $S$ if for every $x\in f^{-1}(S)$, the tangent spaces $T_{f(x)}S$ and $d_{f(x)}T_xX$ together span $T_{f(x)}Y$. In other words,
 \[T_{f(x)}S+d_{f(x)}T_xX=T_{f(x)}Y.\]
 Finally, we will use the following generalized preimage theorem whose proof can be found in \cite{lee1}.
 
 \begin{theorem}[The Preimage Theorem]\label{thm:preimage_thm} Suppose $X$ and $Y$ are smooth manifolds and $S\subseteq Y$ is an embedded submanifold.
 If $f:X\to Y$ is a smooth map that is transverse to $S$, then $f^{-1}(S)$ is an embedded submanifold of $X$ whose codimension is equal to the codimension of $S$ in $Y$.
 \end{theorem}

\section{Voronoi Diagrams in Polyhedral Norms}\label{sec:VDbasic}

In this section, we study the Voronoi diagrams of algebraic varieties with respect to polyhedral norms. We state our main definitions and prove results regarding Voronoi cells of smooth points of algebraic varieties. In what follows, we assume that $B\subset \RR^n$ is a centrally symmetric convex polytope that corresponds to the unit ball of a polyhedral norm on $\RR^n.$

	\begin{definition}
		Let $X\subset \RR^n$ be a variety and let $v\in X$. Given a polyhedral norm $h_\L$, the \emph{Voronoi cell} of $v$ with respect to $h_\L$ is:
		\[\vor_X(v):=\{y\in \RR^n~|~d(y,X)=h_\L(y-v)\}.\]
	\end{definition}

Given a smooth point $v$ on a variety $X$, we define the $\type$ of $v$ as the collection of faces of the unit ball $B$ that have a supporting hyperplane whose normal vector is contained in the normal space of $v$, $N_v.$

\begin{definition}\label{def:type}
	Let $X$ be a variety and $N_v$ 
	be the normal space of $X$ at a smooth point $v\in X.$ We define the \emph{type} of $v$ as:
	$$ \type(v):=\left\{F\in\F(B)~|~C(P,F)\cap N_v
	\neq \emptyset \right\}.$$
\end{definition}

We point out that the definition above differs from the definition of type in~\cite[Proposition 4]{CELIK_Wasserstein}. In our case, the type of a point does not depend on a second point. It depends only on the local geometry of the variety and the polyhedral ball. In \Cref{def:optimizing_face}, 
we will refine the definition of type by defining an \emph{optimizing face} which is closer to the definition of type from \cite{CELIK_Wasserstein}. 

\begin{lemma}
	Given a ball $B$ as defined above and a smooth point $v\in X$, if $F\in \type(v)$, then $-F\in \type(v)$. 
\end{lemma}
\begin{proof}
    Since $B$ is centrally symmetric, it holds that $N_v\cap C(P,F)\neq\emptyset$ if and only if $N_v\cap  C(P,-F)\neq \emptyset$. 
\end{proof}

The next proposition describes the Voronoi cell of a point based on its type. Specifically, we find a union of cones containing the Voronoi cell of a smooth point in a variety. We recall that given a face $F$ of a polyhedral ball $B$, we let $\overline{C(B,F)}$ denote the Euclidean closure of $C(B,F)$ in $\RR^n\setminus \{0\}.$

\begin{proposition}\label{prop:coneandtype}
	Let $X$ be a real algebraic variety and $v\in X$ a smooth point. Then the Voronoi cell 
    $$\vor_X(v) \subset \bigcup_{F\in\type(v)}\overline{C_v(B^*,F^*)}.$$
\end{proposition}

\begin{proof}
	Let $u \in \RR^n\setminus X$ be a point in the Voronoi cell of $v \in X.$ We claim $u-v\in \overline{C(B^*,F^*)}$ for some $F \in \type(v)$. Let us assume this is not true so that $u-v\in C(B^*,T^*)$ for some face $T$ 
	such that $T\not\subseteq F$ for all $F\in \type(v).$ Since $u-v$ is minimized over the face $T$, we have that $B_{\lambda}(u)$ intersects $X$ at the point $v$ which lies in the face $T$ where $\lambda:=d(u,X)$. Hence, the tangent space at $v$ is a supporting hyperplane of $T$. Equivalently, there exists a vector $n_v\in N_v$ such that a hyperplane with normal vector $n_v$ is a supporting hyperplane of the face $T.$ This gives a contradiction.
\end{proof}

	\begin{remark}
		In the Euclidean norm, the Voronoi cell of a smooth point $v$ of an algebraic variety $X$ is contained in the normal space of $v.$ Our result above gives an analogous result for polyhedral norms. The faces of $B$ in $\type(v)$ are the faces of the unit ball that intersect the normal space of $v$ and the Voronoi cell of $v$ is contained in the normal cones of the dual faces corresponding to the faces in $\type(v)$. We call $$\bigcup_{F\in\type(v)}\overline{C_v(B^*,F^*)}$$
        the \textit{Voronoi cone} of $v$ in $X$ and denote it by $\vc(v)$. 
	\end{remark}

Based on \Cref{def:type} and \Cref{prop:coneandtype}, the following theorem gives a way to find all faces in $\type(v)$ where $v$ is a smooth point on a variety of codimension $c<n$.
The smoothness condition 
guarantees that the normal space at the point is of the expected dimension.

	\begin{theorem}\label{thm:generalresult}
		Let $X$ be an algebraic variety of codimension $c<n$, and $N_v$ be the normal space at a smooth point $v\in X$. Then 
		\[\type(v)=\{F\in\F(B) \mid N_v\cap F^*\neq \emptyset \}.\]
		
			\end{theorem}

 \begin{proof}
		We want to prove the equality of the two sets above for a smooth point $v\in X$. 
		Assume that $F\in \type(v)$. By definition, $N_v\cap C(P,F)\neq \emptyset$, so there is a vector $w\in N_v$ that is minimized over the face $F^*$ of $B^*$. Therefore, there exists $\lambda\in \RR$ such that $\lambda w\in F^*$ for $F\in\type(v)$ and hence, $\lambda w\in N_v\cap F^*$. 
		
		On the other hand, let $w\in N_v$ such that there exists $\lambda>0$ for which $\lambda w\in F^*$ for some face $F^*$ of $B^*.$ This implies that $-w$ is minimized over $F^*$ and hence, $-w\in C(B,F)$. Consequently, $N_v\cap C(B,F)\neq \emptyset$ and $F\in \type(v).$
	\end{proof}

\begin{figure}
	\centering
	\begin{subfigure}[b]{0.4\textwidth}
		\includegraphics[scale=0.23]{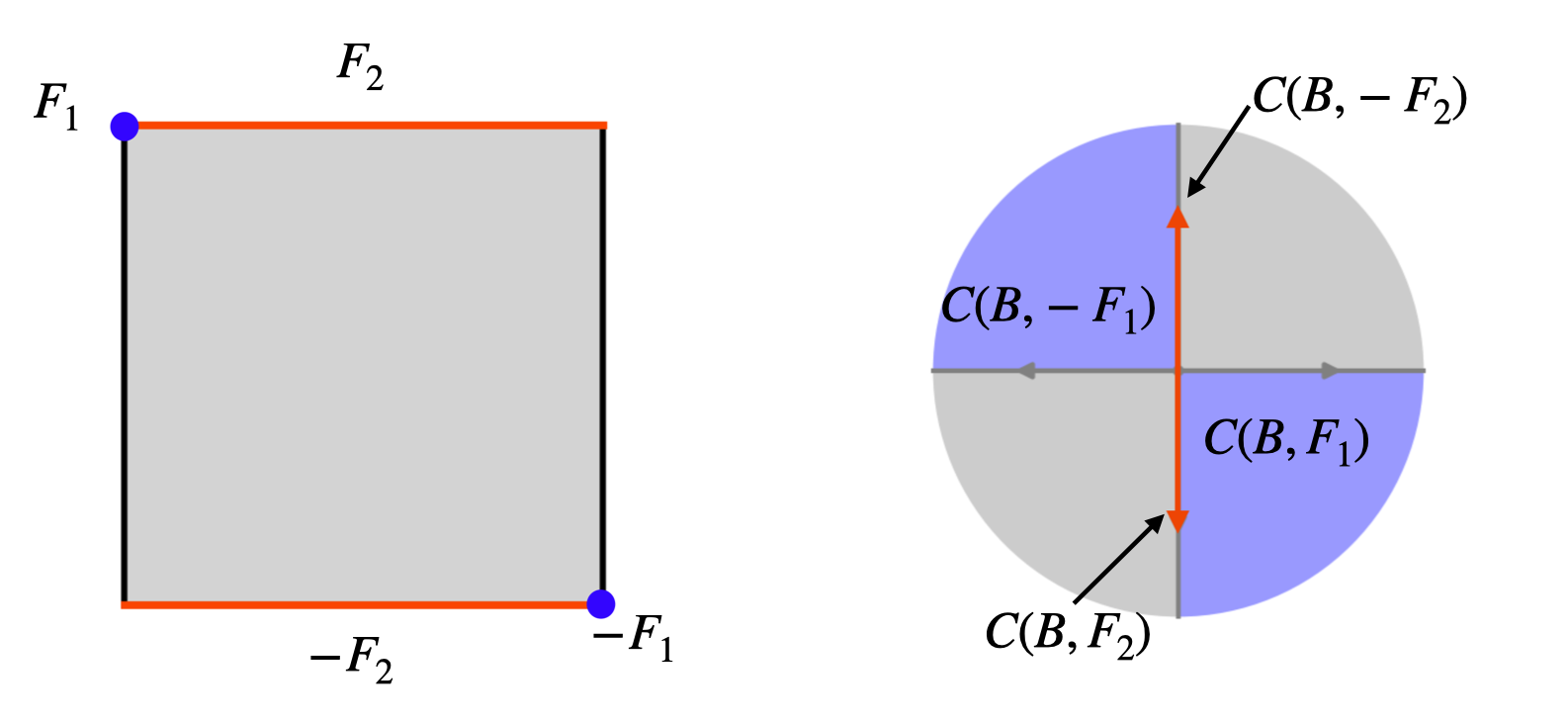} \hfill
		\caption{ Unit square $B$ and $\Sigma_B$ 
		}
		\label{fig:unit-square-normal-fan}
	\end{subfigure}
	\hspace{2cm}
	\begin{subfigure}[b]{0.4\textwidth}
		\includegraphics[scale=0.13]{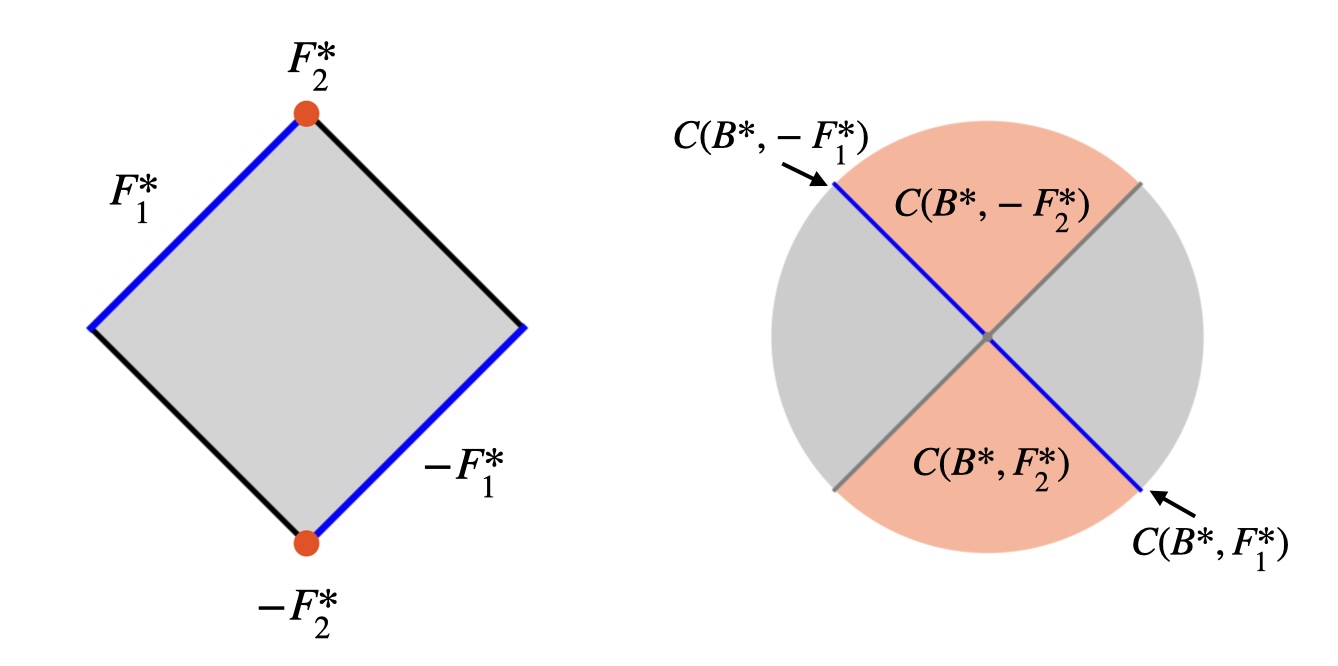} \hfill
		\caption{ $B^*$ and $\Sigma_{B^*}$  }
		\label{fig:unit-diamond-normal-fan}
	\end{subfigure}
	\newline
	\begin{subfigure}[b]{0.2\textwidth}
		\includegraphics[scale=0.2]{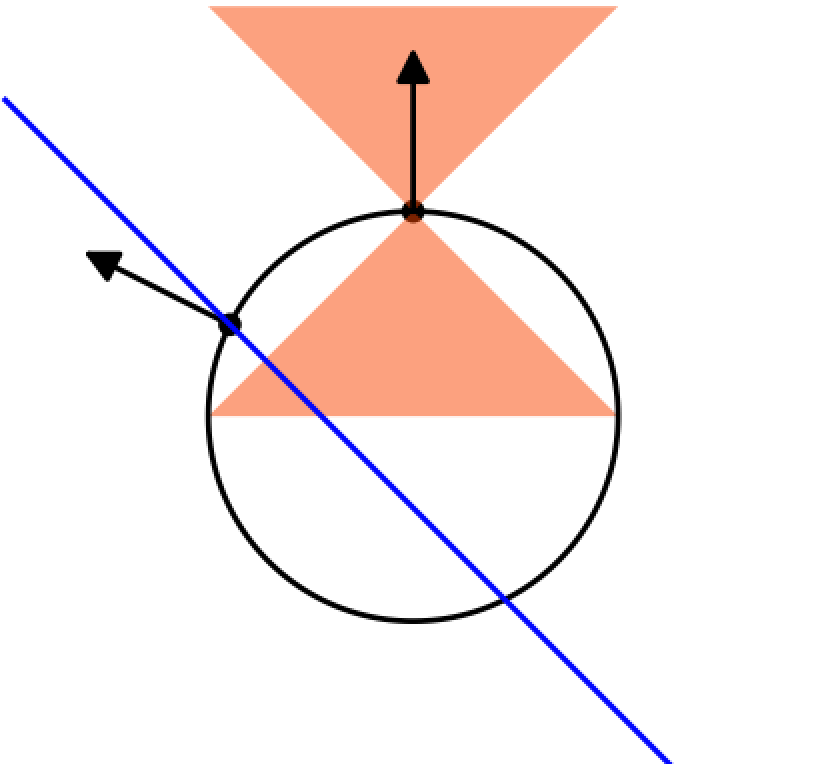}
		\caption{$V_1$}
		\label{fig:type_circle}
	\end{subfigure}
	\hfill
	\begin{subfigure} [b]{0.2\textwidth}
		\includegraphics[scale=0.21]{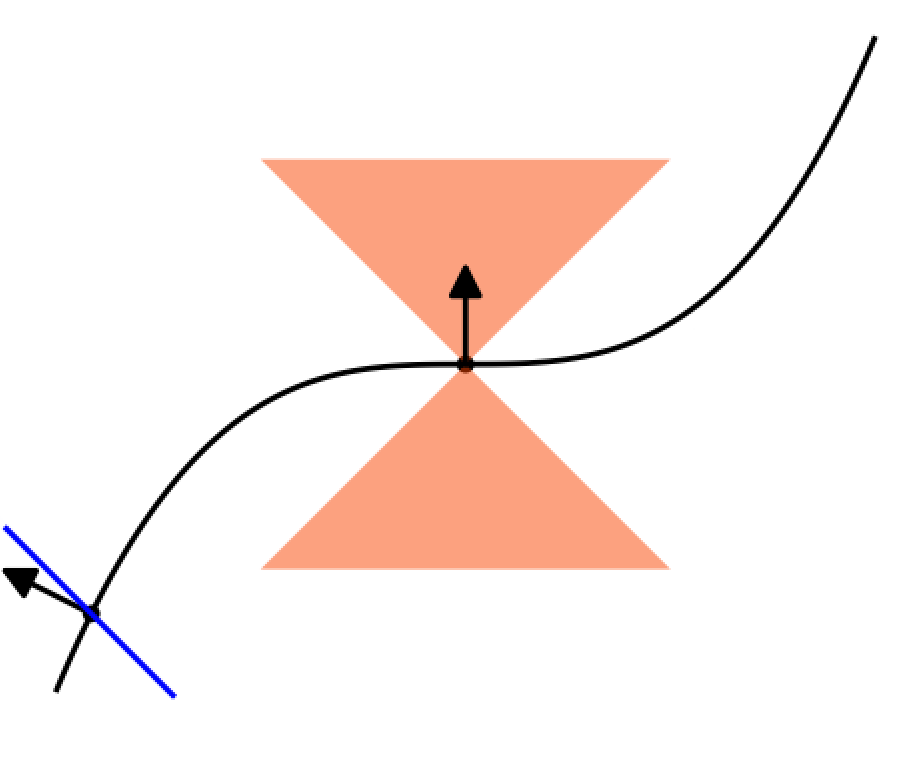}
		\caption{$V_2$}
		\label{fig:type_cubic}
	\end{subfigure}
	\hfill
	\begin{subfigure}[b]{0.2\textwidth}
		\includegraphics[scale=0.18]{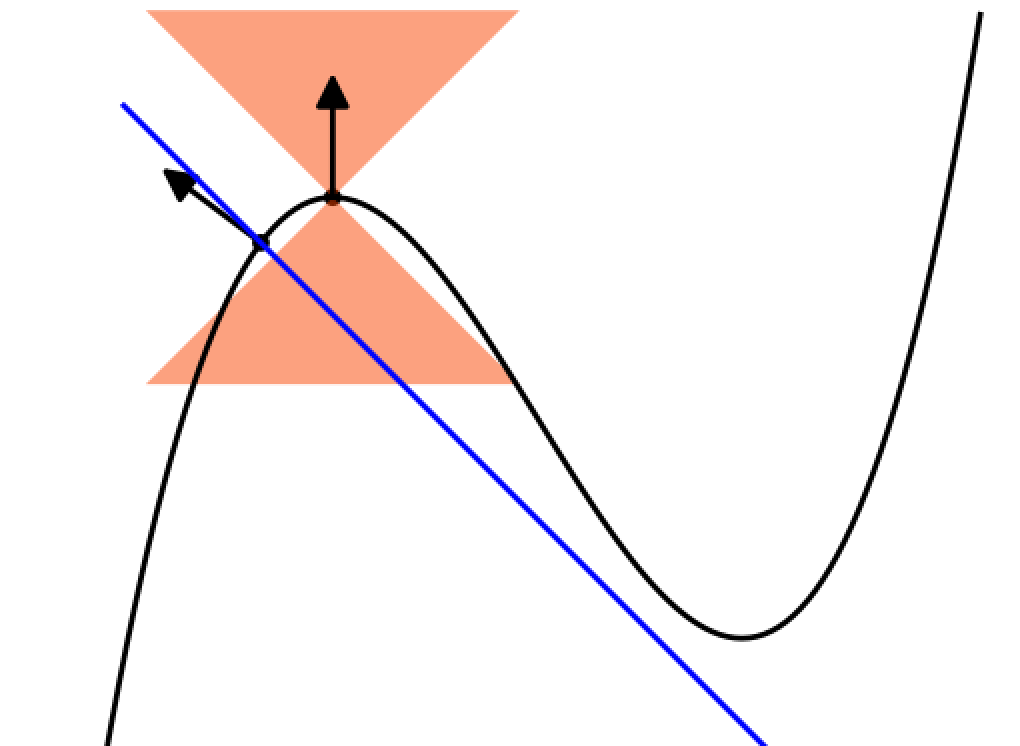}
		\caption{$V_3$}
		\label{fig:type_cubic_modes}
	\end{subfigure}
	\hfill
	\begin{subfigure}[b]{0.2\textwidth}
		\includegraphics[scale=0.25]{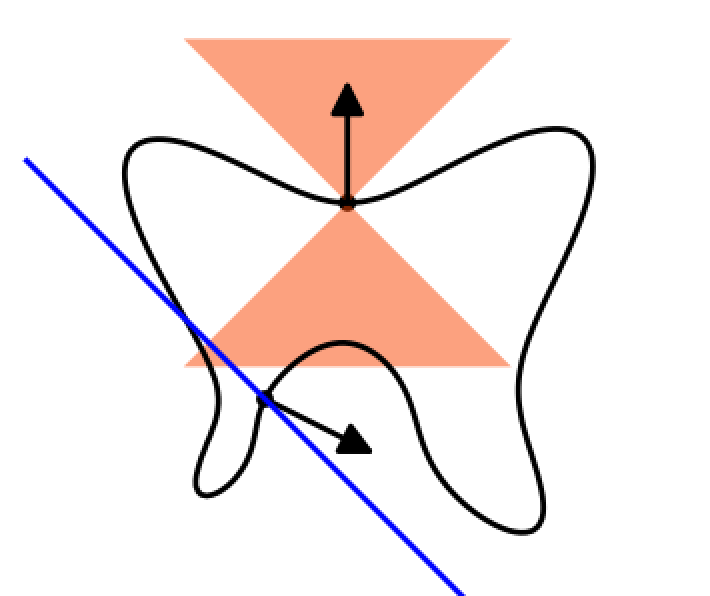}
		\caption{$V_4$}
		\label{fig:type_tooth}
	\end{subfigure}
	\caption{Illustration of $\type(v)$ in \Cref{ex:type_plane_curves}.}
\end{figure}

\begin{remark} \label{rmk:type2faces}
	For codimension-one varieties, any smooth point $v$ in the variety has a one-dimensional normal space. Therefore, it follows from \Cref{def:type} and \Cref{thm:generalresult} that $\type(v)$ 
	has only two symmetrically opposite faces.  
\end{remark}

Given a smooth codimension-one variety $X$ and a smooth point $v\in X$, one can easily compute the type of $v$ for a given polyhedral norm in \textsc{SageMath} as illustrated in the following example.

\begin{example}
Let the unit square in \Cref{ex:plane_curve} be the unit ball. The linear functionals defining this unit ball are given by $\{\pm x, \pm y\}$. Let $X = \V(f)$ be the zero set of the parabola $f=y-x^2$ and $v=(2,4)$ the chosen smooth point. The set $\type(v)$ can be computed as follows:

 \begin{verbatim}
f(x1,x2) = x2-x1^2 
v = (2,4)
L = [[0,1],[0,-1],[1,0],[-1,0]] 
L1 = [ [1] + s for s in L ]
B = Polyhedron(ieqs = L1)
NF = NormalFan(B)
grad = f.diff()
R1 = NF.cone_containing(grad(*v))
R2 = -R1
def equality(cone1, cone2):
    return set(tuple(v) for v in cone1.rays()) == set(tuple(v) for v in cone2.rays())
Type = []
for f in B.faces(B.dim()-R1.dim()):
    if equality(f.normal_cone(), R1) or equality(f.normal_cone(), R2):
        Type.append(f)
if Type:
    print("Type of v:")
    for face in Type:
        print("Face:", face)
        print("Vertices:", [v.vector() for v in face.vertices()])
        print("---")
\end{verbatim}

In the above code display, $\texttt{L}$ is the set of linear functionals, $\texttt{B}$ is the corresponding polyhedral unit ball and the code returns a list of the faces in $\type(v)$. $\texttt{R1}$ is the cone in the normal fan that contains the gradient of $f$ and $\texttt{R2}$ is the negative of $\texttt{R1}$. 
The list \texttt{Type} contains the faces of \texttt{B} whose inner normal cones are $\texttt{R1}$ and $\texttt{R2}$, and each face is given in terms of its vertices. The code is easily usable for codimension-one varieties in any dimension.

\end{example}

\begin{example}\label{ex:type_plane_curves}
Let $B$ be the unit square and $B^*$ its dual as in \Cref{ex:plane_curve}. Consider the plane curves 
	\begin{align*}
		V_1:\text{ } & x^2+y^2-1=0,&\\
		V_2:\text{ } & x^3-y=0,& \\
		V_3:\text{ } &y-x^3/20^3-10x^2/20^2-x/20+1=0, &\\
		V_4:\text{ } & x^4 - x^2 y^2 + y^4 - 4 x^2 - 2 y^2 - x - 4 y + 1= 0. &\end{align*}
     Let $F_1$ be the vertex and $F_2$ be the edge of $B$ depicted in blue and orange in \Cref{fig:unit-square-normal-fan}.
     On each of the curves in \Cref{fig:type_circle,fig:type_cubic,fig:type_cubic_modes,fig:type_tooth}, we consider two points $v_1$ and $v_2$ whose normal spaces intersect the inner normal cones of the vertex and the edge respectively. For these points, $\type(v_1)=\{F_1,-F_1\}$ and $\type(v_2)=\{F_2,-F_2\}$ respectively, and by \Cref{prop:coneandtype}, $\vor(v_1)\subseteq \overline{C(B^*,F_1^*)}\cup \overline{C(B^*,-F_1^*)}$ and $\vor(v_2)\subseteq \overline{C(B^*,F_2^*)}\cup \overline{C(B^*,-F_2^*)}.$ The plotted points and their types are summarized in \Cref{tab:Ex_types_planecurve}. The cones containing the Voronoi cell of each point are depicted in \Cref{fig:type_circle,fig:type_cubic,fig:type_cubic_modes,fig:type_tooth}.
     \begin{table}[h]
         \centering
         \begin{tabular}{c|c|c}
             Curve & Point $v_1$ with type $\{\pm F_1\}$ & Point $v_2$ with type $\{\pm F_2\}$  \\
             \hline
              $V_1$ & $\left(-\frac{1}{\sqrt{2}},\frac{1}{\sqrt{2}}\right)$ & $(0,1)$\\
              $V_2$ & $\left(\frac{-40}{3}\sqrt{30},\frac{-80}{9}\sqrt{30}\right)$ & $(0,0)$ \\
              $V_3$ & $\left(-\frac{20}{3}\sqrt{177} - \frac{200}{3}, \frac{38}{9}\sqrt{177} + \frac{1883}{27} \right)$ & $\left(-\frac{20}{3}\sqrt{97} - \frac{200}{3}, \frac{194}{27}\sqrt{97} + \frac{1883}{27}\right)$ \\
              $V_4$ & $(-1.070485058466869, -0.4474900924702774)$ & $(-0.06430591417681053, 1.945149406467458)$
         \end{tabular}
         \caption{The points $v_1$ and $v_2$ from \Cref{ex:type_plane_curves} along with their types.}
         \label{tab:Ex_types_planecurve}
     \end{table}
     
 \end{example}

\begin{example} \label{ex:space_curve_type}

Consider the twisted cubic $X$ in $\RR^3$ defined as the vanishing locus of the polynomials $f_1=xy-z$, $f_2=xz-y^2$ and $f_3=y-x^2$ (see \Cref{fig:twisted_cubic_type}). We take the cube as the unit ball $B,$ and the octahedron as its dual $B^*.$

\begin{figure}[h]
	\centering
	\includegraphics[scale=0.3]{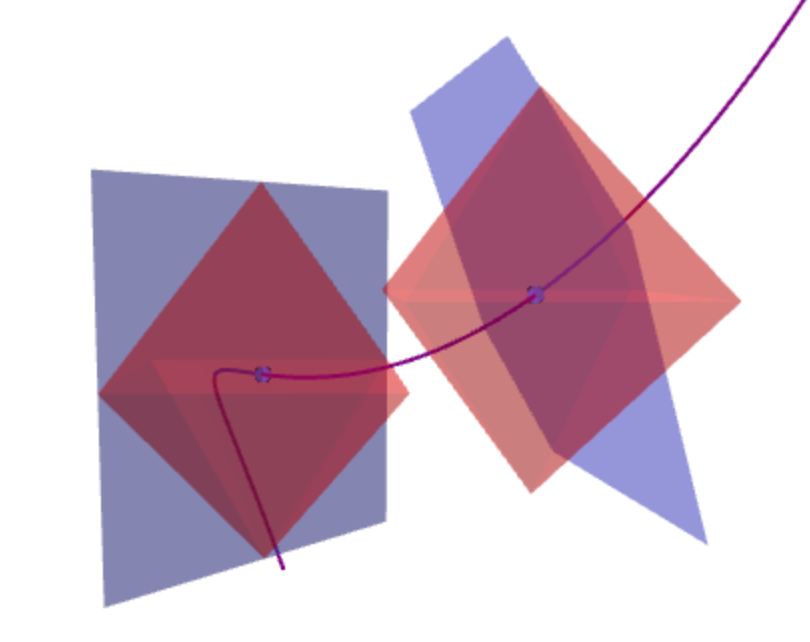}
	\caption{Twisted cubic with the normal space depicted at $v_1=(0,0,0)$ and $v_2=(0.6,0.6^2, 0.6^3).$}
	\label{fig:twisted_cubic_type}
\end{figure}
	The normal space of $X$ at $v_1=(0,0,0)$ is the green plane given by $x=0$ in \Cref{fig:normalfan_origin_tcubic}, which intersects the cones corresponding to four edges and four facets. The faces of $B$ associated with these cones are precisely the set $\type(v_1)$. The cones of the dual fan that contain the Voronoi cell of $v_1$ are depicted in blue in \Cref{fig:dualfan_tcubic_origin}.

	Now consider the point $v_2=(0.6,0.6^2,0.6^3)\in X$.
	The normal space at $v_2$ is given by $x+1.2 y+1.08 z=0$ and is shown in green in \Cref{fig:normalfan_tcubic_other}. This plane intersects the normal cone corresponding to six vertices and six edges which determine the $\type(v_2)$. Note that the closure of the two full-dimensional red cones is disjoint from the normal space and is the union of cones of faces in the set $\F \setminus \type(v_2)$. By
	\Cref{thm:generalresult}, $\vor_X(v_2)$ is contained in the closure of the union of the green two-dimensional cones and gray full-dimensional cones in \Cref{fig:dualfan_tcubic_other}.

\begin{figure}
	\centering
	\begin{subfigure}[t]{0.22\textwidth}
		\centering 
		\includegraphics[scale=0.2]{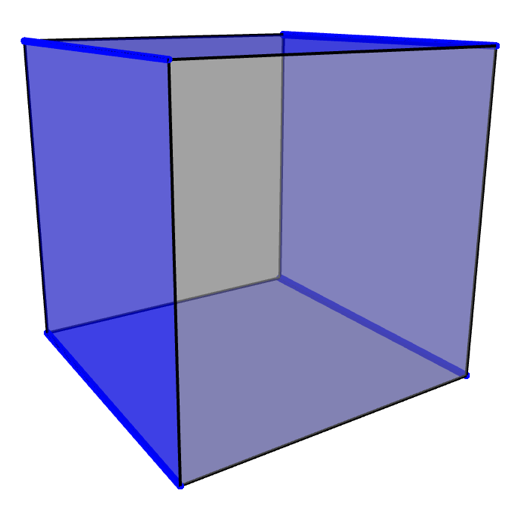}
		\caption{$B$}
		\label{fig:cube_tcubic}
	\end{subfigure}
	\begin{subfigure}[t]{0.22\textwidth}
		\centering
		\includegraphics[scale=0.25]{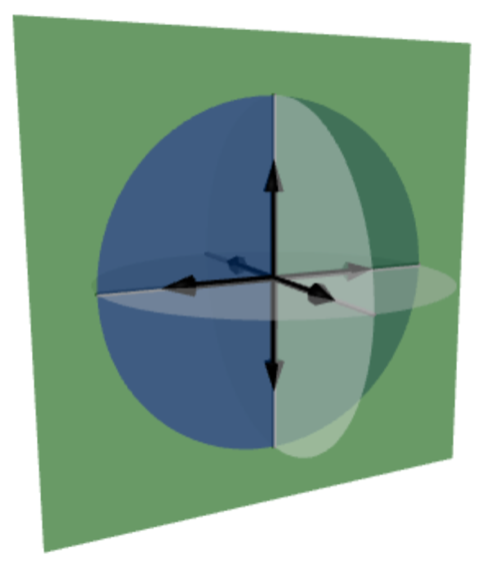}
		\caption{$\Sigma_B$ in gray and blue. In green, the normal space of $X$ at $v_1$. }
		\label{fig:normalfan_origin_tcubic}
	\end{subfigure}
	\begin{subfigure}[t]{0.22\textwidth}
		\centering
		\includegraphics[scale=0.2]{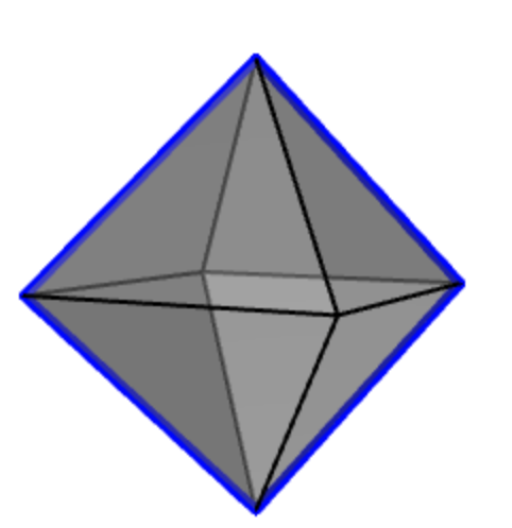}
		\caption{$B^*$ }
		\label{fig:octahedron_tcubic}
	\end{subfigure}
	\begin{subfigure}[t]{0.26\textwidth}
		\centering
		\includegraphics[scale=0.25]{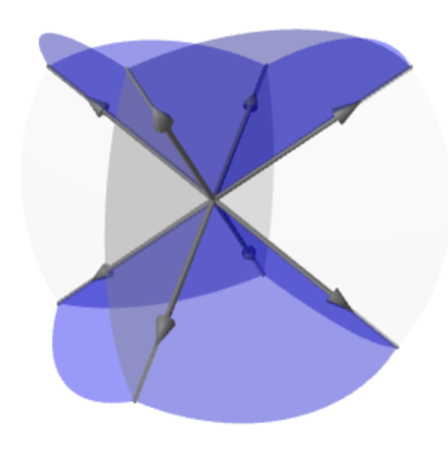}
		\includegraphics[scale=0.25]{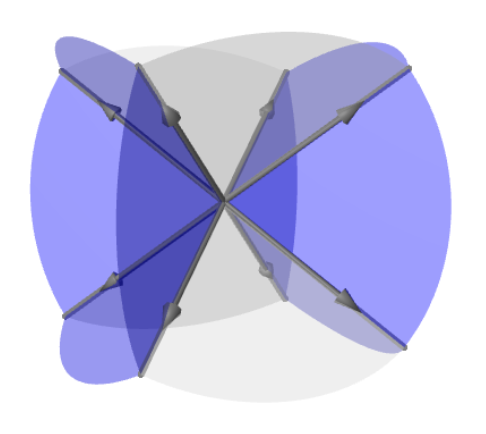}
		\caption{$\Sigma_{B^*}.$ In blue, the cones whose closure contains $\vor_X(v_1).$ }
		\label{fig:dualfan_tcubic_origin}
	\end{subfigure} \\
	\begin{subfigure}[t]{0.22\textwidth}
		\centering
		\includegraphics[scale=0.18]{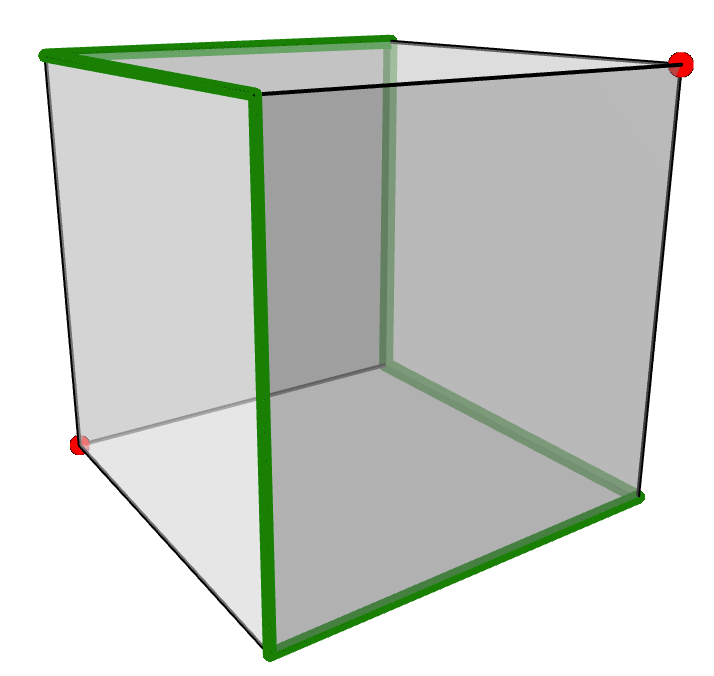}
		\caption{$B$}
		\label{fig:cube_tcubic_other}
	\end{subfigure}
	\begin{subfigure}[t]{0.22\textwidth}
		\centering
		\includegraphics[scale=0.2]{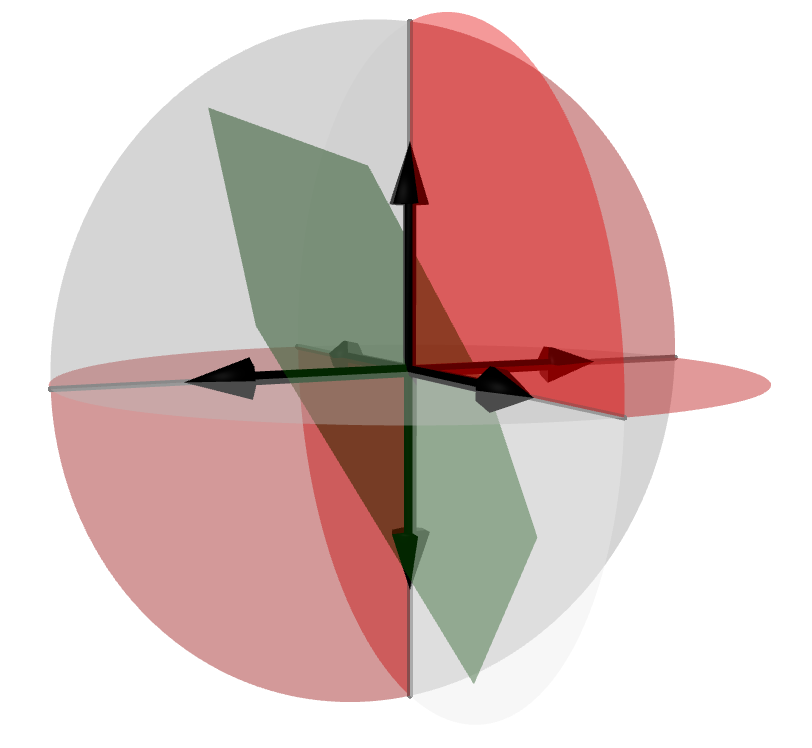}
		\caption{$\Sigma_B.$ The green plane is the normal space to $X$ at $v_2$.}
		\label{fig:normalfan_tcubic_other}
	\end{subfigure}
	\begin{subfigure}[t]{0.22\textwidth}
		\centering
		\includegraphics[scale=0.25]{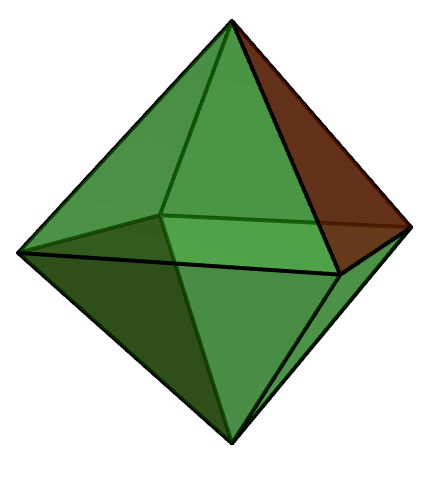}
		\caption{$B^*$ }
		\label{fig:octahedron_tcubic_other}
	\end{subfigure}
	\begin{subfigure}[t]{0.25\textwidth}
		\centering
		\includegraphics[scale=0.25]{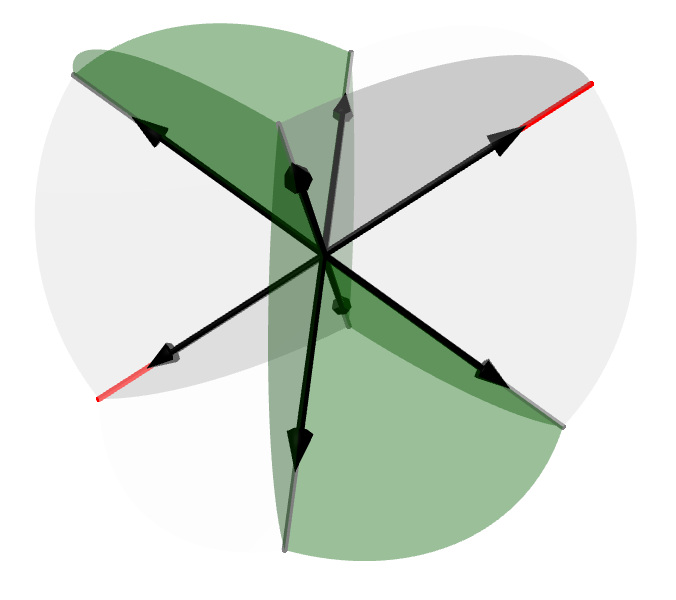}
		\caption{$\Sigma_{B^*}.$ In green and gray, the cones whose closure contains $\vor_X(v_2).$ }
		\label{fig:dualfan_tcubic_other}
	\end{subfigure}           
	
	\caption{Figures from \Cref{ex:space_curve_type}. The polyhedral ball $B$ is the cube and the variety $X$ is the twisted cubic. The top row illustrates the type of $v_1=(0,0,0)$ and the bottom row illustrates the type of
	 $v_2=(0.6,0.6^2,0.6^3)$.  
     $\type(v_1)$ is given by the blue faces of $B$ in (a) and $\type(v_1)$ is given by the green faces of $B$ in (e).         The corresponding dual faces to those in $\type(v_1)$ and $\type(v_2)$ are depicted in blue and green respectively in (c) and (g).
    The normal spaces of $v_1$ and $v_2$ are shown in green in (b) and (f) and their intersection with the normal fan of $B^*$ is given in (d) and (h) respectively.
        }
	\label{fig:space_curves}
\end{figure}

\end{example}

Let $v \in X$ be a smooth point and define $\overline{\type(v)}$ to be the closure of $\type(v)$. In particular,
$$\overline{\type(v)}:=\{  F\in \F(B) \mid F \subseteq F' ~~\mbox{for some}~~
F'\in \type(v)\}.$$
Furthermore, for given $u\in \RR^n\setminus X$ and $\lambda>0$, denote the face in which $B_\lambda(u)$ intersects $X$ at $v$ as $F_\l(u,v).$ 
We will now use this to define optimizing faces of a point $v\in X$.

\begin{definition}\label{def:optimizing_face}
	Let $X$ be an algebraic variety and $F$ be a face of the unit ball $B$ associated to a polyhedral norm. We say that $F$ is an \textit{optimizing face} for $v\in X$ if 
	\begin{enumerate}[label=(\roman*)]
		\item 
		$F\in \overline{\type(v)}$,
		\item there exists $u\in \RR^n\setminus X$ such that $d(u,X)=d(u,v)=\lambda$, 
		\item $F_\lambda(u,v)=F$. 
	\end{enumerate}
	Moreover, if $F\in \type(v)$ we say that $F$ is a \emph{proper optimizing face}.
\end{definition}

It may seem counterintuitive that an optimizing face of $v \in X$ may not lie in $\type(v)$. We refer the reader to \Cref{ex:optimizing face and types} for an explanation of this.

The definition of optimizing face is similar to the definition of type in~\cite[Proposition 4]{CELIK_Wasserstein}. In our case, we add the additional condition that the optimizing face is in $\overline{\type(v)}.$ Furthermore, we will say $F$ is an optimizing face of $v$ with respect to $u$ if  $F_\lambda(u,v)=F$ and $d(u,X)=\lambda.$

\begin{remark}\label{rem:conedim}
	 For completeness, we highlight how the notation and results in \cite{BECEDAS2024102229} compare to ours. 
    To describe the Voronoi cell of a point on a variety, the authors in \cite{BECEDAS2024102229} use the notion of a \textit{face cone}, which is equivalent to our use of the inner normal cone of the dual face.
    Using \Cref{prop:coneandtype}, one can compute the upper bound on the dimension of $\vor_X(v)$ as
    \begin{equation}
    	\max_{F\in \type(v)}\dim(\overline{C_v(B^*,F^*)}) .
    \end{equation}
    This bound is the upper bound provided in Theorem 3.6 in \cite{BECEDAS2024102229}, where the authors use affine spaces associated to each face $F\in \F(B)$ to decide which cones can contain the Voronoi cell. 

    Furthermore, Lemma 3.5 in \cite{BECEDAS2024102229} guarantees that if $v\in X$ is the closest point to $u\in \mathbb{A}^n \setminus X$, then the ball centered in $u$ that realizes the distance from $u$ to $X$ will intersect $X$ at $v$ in a face $F$ non-transversely. In our setting, $\overline{\type(v)}$ collects these faces for real varieties. That is, $\overline{\type(v)}$ contains the faces where $C(B^*,F^*)$ can potentially contribute to $\vor_X(v).$

\end{remark}

\begin{example}\label{ex:optimizing face and types}
	In \Cref{ex:type_plane_curves}, we saw that for $v_2=(0,1)\in V_1 = \V(x^2 + y^2 -1)$, $\type(v_2)=\{F_2,-F_2\}$. 
	However, note that for the point $u=(-1,2)$, the optimizing face $F_\l(u,v)$ is the vertex $-F_1$ contained in $-F_2$ where $\l=1$. This is consistent with \Cref{rem:conedim} and with \cite[Lemma 3.5]{BECEDAS2024102229} because the affine hull of $F_2$ translated to $v_2$ intersects $V_1$ non-transversely.
	For the point $v_1=\frac{1}{\sqrt{2}}(-1,1)$ in $V_1$, $\type(v_1)=\{F_1,-F_1\}$, and both $F_1$ and $-F_1$ are optimizing faces and hence, proper optimizing faces. This can be easily verified for points $u_1=(-2,2)$ and $u_2=\frac{1}{2\sqrt{2}}(-1,1).$ In the context of \Cref{rem:conedim}, note that the affine hull of $F_2$ translated to $v_1$ is the line defined by equation $y-\frac{1}{\sqrt{2}}=0$, which intersects $V_1$ at $v_1$ transversely hence, by Lemma 3.5 in \cite{BECEDAS2024102229}, $\vor(v_1)$ does not contain any point in $C_{F_2}.$

	For the curve $V_2 = \V(x^3 - y)$  from~\Cref{ex:type_plane_curves}, $\type((0,0))=\{F_2,-F_2\}$. However, since the affine hull of $F_2$ translated to the origin intersects $V_2$ transversely, by \cite{BECEDAS2024102229} we have that the Voronoi cell of $(0,0)$ does not contain any point in $C_{F_2}$.
    For every other $v \in V_2$, the vertices $F_1$ and $-F_1$ are proper. 
\end{example}

For a smooth point $v \in X$, $\type(v) \neq \emptyset$; 
however, $\vor_X(v)$ could contain only the point $v$ itself. This occurs when the set of optimizing faces is empty. Such a behavior is typical for saddle points of varieties in higher dimensions for specific norms. For instance, if in \Cref{ex:optimizing face and types} we consider the polyhedral ball to be generated by $\L=\left\{ \pm(x+y), \pm( -x+y)\right\}$, then the set of optimizing faces of $(0,0)$ is nonempty. This type of behavior is captured in \cite[Lemma 3.5]{BECEDAS2024102229}. 
We provide an example of this behavior for a variety in $\RR^3$ below. 

\begin{example}\label{ex:hyperbolic_example}
    Consider the hyperboloid $X = \V(36x^2+9y^2-4z^2-36) \subset \RR^3$ and take the three-dimensional cube as the unit ball (see \Cref{fig:cube-octa-fans}). The normal space of $v = (1,0,0)$ is the $x$-axis, so $\type(v)$ contains the two facets of the cube that are orthogonal to the $x$-axis. Note, however, that the set of optimizing faces of $v$ is empty because at $v$ there is a saddle point so there is no point outside of $X$ whose distance to $X$ is minimized by $v$.

    This situation is also a particular case of \cite[Lemma 3.5]{BECEDAS2024102229} with $F$ any facet in $\type\left(v \right).$ Indeed, since the hyperplane $x-1=0$ and the hyperboloid intersect transversely at $v$ (both considered as subvarieties of $\RR^3$), then the Voronoi cell of $v = (1,0,0)$ does not contain any point in $C(B^*,F^*).$  

    \begin{figure}[h]
        \centering
        \includegraphics[scale=0.3]{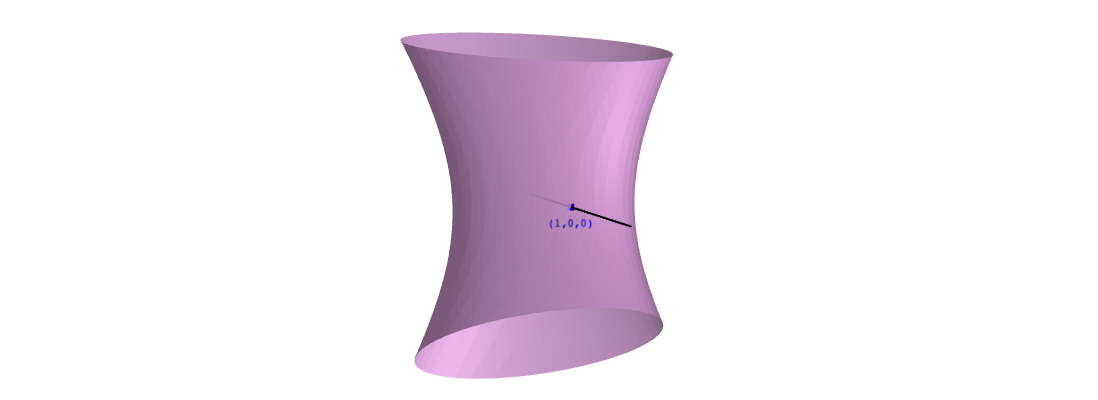}
        \caption{The hyperboloid defined as $\V(36x^2+9y^2-4z^2-36)$. The point $v = (1,0,0)$ is a saddle point. Its Voronoi cone is the $x$-axis depicted in black, but its Voronoi cell consists of the single point~$v$. }
        \label{fig:hyperbolic_example}
    \end{figure}

\end{example}

\section{Stratification of Codimension-One Manifolds}\label{sec:Codim1}

For codimension-one varieties, the normal space at a smooth point is spanned by a 
unique vector (up to scaling) and hence is one dimensional. Consequently, the type
of such a point is a set of two opposite faces of the same dimension (\Cref{rmk:type2faces}). 
Using this description, we can stratify the smooth part of a variety based on the expected dimension of the Voronoi cell. In particular, we expect the dimension of the Voronoi cell to be equal to the dimension of the Voronoi cone, so
we stratify a codimension-one variety by the dimension of the Voronoi cone defined by a fixed polyhedral norm with unit ball $B$. 
To that end, we will use the basic notions of topology (\Cref{sec:top})
to stratify $\RR\PP^{n-1}$ (\Cref{thm:strat-sphere}) and general 
codimension-one varieties (\Cref{thm:stratgeneral}).

Let $B$ be a polyhedral ball defined by a polyhedral norm $h_{\mathcal{L}}$, and let $\Sigma_B=\bigcup_{F\in \F}\{C_F\}$ denote its inner normal fan. 
Consider the following map
\begin{align}\label{eq:sphereproj}
\varphi:\RR^n\setminus \{0\} \to \RR\PP^{n-1}.
\end{align}
This is the projection of $\RR^n\setminus \{0\}$ onto the real projective space $\RR\PP^{n-1}$; it maps each vector to its equivalence class. The map $\varphi$ is surjective and continuous. The space  $\RR\PP^{n-1}$ is endowed with quotient topology induced by $\varphi$ and thus a set $U \subset \Rp^{n-1}$ is open if $\varphi^{-1}(U)$ is an open set in $\RR^{n}\setminus\{0\}.$ Note that $\varphi(C_F)=\varphi(C_{-F})$ and, if $U$ is open, then $\varphi^{-1}(U)$ is an open double cone, i.e. the union of two opposite open cones. 

\begin{lemma}\label{lem:disjointimage}
	Let $F,F'\in \F$, and $ F'\notin \{F,-F\}$. Then $\varphi({C_F})\cap \varphi({C_{F'}}) = \emptyset$.
\end{lemma}
\begin{proof}
	Assume to the contrary that $x\in \varphi({C_F})\cap \varphi({C_{F'}})$. Then $\varphi^{-1}(x)\subseteq (C_F\cup C_{-F})\cap (C_{F'}\cup C_{-F'})$. This implies that $\varphi^{-1}(x)$ is an empty set. Since $\varphi$ is a surjective map, we get a contradiction.
\end{proof}

Let $[n]\coloneqq \{0,\ldots ,n-1\}$. For $i\in [n]$, recall that $\F_i$ is the set of all the faces of dimension $n-i-1$. We use this to define the following sets:
\begin{equation*}
	S_i=\bigsqcup_{F\in \F_i} ~~\varphi({C_F}).
\end{equation*}

Note that by construction, each set $S_i$ is a set of dimension $i.$ In the next result, we prove that the partition of sets above stratifies $\RR\PP^{n-1}.$

\begin{theorem}[Stratification of $\RR\PP^{n-1}$]\label{thm:strat-sphere}
	Let $B$ be the unit ball associated to a polyhedral norm, and $\Sigma_B$ its inner normal fan. The collection $\mathcal{S}_B=\{S_i\}_{i=0}^{n-1}$ stratifies $\RR\PP^{n-1}.$
\end{theorem}	

\begin{proof}
	We will show that the collection $\mathcal{S}_B$ satisfies the properties \ref{def212:it1}-\ref{def212:it5} of~\Cref{def:Strat_top_space}.
	
	To show \ref{def212:it1}, we note that $\bigcup_{F\in \F}\{C_F\}=\RR^{n}\setminus\{0\}$. Hence, by surjectivity of $\varphi$, we have $\Rp^{n-1}=\bigcup_{i\in [n]}S_i.$
For all $i\in [n]$, $S_i$ is a smooth manifold with respect to the subspace topology of $\Rp^{n-1}$ and $S_i\hookrightarrow \Rp^{n-1}$ is a smooth embedding. Hence, $S_i$ is an embedded submanifold of $\Rp^{n-1}$ for all $i$. This shows \ref{def212:it2} holds.
	
	We now show \ref{def212:it3}. Since $\varphi$ is open and continuous, we have $\overline{\varphi(C_F)}=\varphi(\overline{C_F})$ for all $F\in \F$. Consequently, $\overline{S_i}=\bigsqcup_{F\in \F_i} \varphi(\overline{C_F})$ and by \Cref{rem:coneclosure}, we can rewrite this as $\overline{S_i}=S_0\sqcup S_1 \ldots \sqcup S_{i}$. Note that $\overline{S_i}\setminus S_i=\overline{S_{i-1}}$ for $i\geq 1$ which is a closed set in $\Rp^{n-1}$ and for $i=0$, $\overline{S_0}\setminus S_0$ is an empty set which is also a closed set. Therefore, $S_i$ is locally closed for all $i\in [n]$.
	
	For item \ref{def212:it4}, assume $i\neq j$. Then, $\F_i \cap \F_j=\emptyset$ and hence, $C_{F_i}\cap C_{F_j}=\emptyset$ for all $F_i\in \F_i$ and $F_j\in \F_j.$ By using \Cref{lem:disjointimage}, it follows that $S_i\cap S_j=\emptyset$.
	
	Finally, we prove \ref{def212:it5}. Using $\overline{S_j}=S_0\sqcup\cdots \sqcup S_j$ and from  \ref{def212:it4} $S_k\cap S_j=\emptyset$ whenever $k\neq j$, we obtain $S_i\cap \overline{S_j}=S_i$ if $i\leq j$ and $S_i\cap \overline{S_j}=\emptyset$
	if $i>j$. Thus, if $S_i\cap \overline{S_j}\neq \emptyset$, then $i\leq j$ and $S_i\subset \overline{S_j}$.
\end{proof}

Let $X \subset \RR^n$ be a smooth real codimension-one algebraic variety. Let $\varphi$ be the map defined in \eqref{eq:sphereproj}. Consider the following set:
\begin{align*}
\mathcal{N}(X)=\{(v,n_v)\in \RR^n\times \RR\PP^{n-1} \mid v\in X \text{ and } N_v(X)=\varphi^{-1}(n_v)\cup \{0\}\},
\end{align*}
where $N_v(X)$ is the normal space of $X$ at $v$. Let $\psi$ be the projection map from $\N(X)$ to $\Rp^{n-1}$ given by $(v,n_v)\mapsto n_v.$ This map may not necessarily be surjective. We further denote by $\pi_1$ and $\pi_2$ the projection of $\N(X)$ onto $X$ and $ \psi(\N(X))$ respectively. Note that $\N(X)\subset X\times \RR\PP^{n-1}$ and both $\N(X)$ and $\psi(\N(X))$ are endowed with the subspace topology. Both $\pi_1$ and $\pi_2$ are continuous maps. Furthermore, the map $\pi_1$ is bijective with $\pi_1(\N(X))=X$ and $\pi_2$ is surjective by construction. 

Let $B$ be a unit ball for some polyhedral norm and  $\S_B:=\{\tilde{S}_i\}_{i=0}^{n-1}$ be the corresponding stratification of $\Rp^{n-1}.$ Let $S_i:=\tilde{S}_i\cap \psi(\N(X))$ for $i=0,\ldots, n-1$. For $i = 0,\ldots, n-1$,  we define the following sets: 
\begin{align}\label{eq:def-Xi}
	X_i:=\pi_1\circ \pi_2^{-1}(S_i).
\end{align}
In \Cref{thm:stratgeneral}, we will show that $\mathcal{X}_B:=\{X_i\}_{i=0}^{n-1}$ defines a stratification of $X.$ To that end, we consider the following map:
\begin{align}\label{eq:map-phi}
	\phi:=\pi_2\circ \pi_1^{-1}:X\to \psi(\N(X)).
\end{align}

\begin{lemma}\label{lem:composedmap}
	The map $\phi$ in \eqref{eq:map-phi} is continuous and an open map.
\end{lemma}

\begin{proof}
    Since $\pi_1$ is bijective, $\phi$ is well defined. 
    Since $\N(X)$ is endowed with the subspace topology of $X\times \Rp^{n-1}$, an open set of $\N(X)$ is of the form $(U\times V)\cap \N(X)$, for some open sets $U\subset X$ and $V\subset \Rp^{n-1}$.
    Note that $\pi_1((U\times V)\cap\N(X))=U$ which is an open set in $X$. Therefore, $\pi_1^{-1}$ is a continuous map. Now, the continuity of $\phi$ follows directly from the continuity of $\pi_1^{-1}$ and $\pi_2$. 
	
	To show that $\phi$ is an open map, it is enough to show that $\pi_1^{-1}$ and $\pi_2$ are open maps. Since $\pi_1$ is bijective and $\pi_1^{-1}$ is continuous, $\pi_1^{-1}$ is an open map. 
	The map $\psi:\N(X)\to \Rp^{n-1}$ is a projection map and hence, an open map. By \Cref{lem:restr_proj}, it is straightforward to verify that $\pi_2$ is an open map. 
\end{proof}

\begin{theorem}\label{thm:stratgeneral}
	Let $B$ be a unit ball with respect to a polyhedral norm and let $\S_B:=\{\tilde{S}_i\}_{i=0}^{n-1}$ be the stratification of $\Rp^{n-1}$. Let $X$  be a real, smooth, connected, codimension-one manifold without boundary. If $\phi$ in \eqref{eq:map-phi} is transversal to $S_i$ for all $i=0,\ldots n-1$, then $\mathcal{X}_B$ stratifies $X$ with respect to $B$.
\end{theorem}

\begin{proof}
	We claim that $\mathcal{X}_B$ satisfies \ref{def212:it1}-\ref{def212:it5} of \Cref{def:Strat_top_space} and therefore stratifies $X$. 
	
	Since $\Rp^{n-1}=\cup_{i=0}^{n-1}\tilde{S}_i$ by \Cref{thm:strat-sphere}, we get $\psi(\N(X))=\cup_{i=0}^{n-1}S_i$. Hence, \ref{def212:it1} holds by \eqref{eq:def-Xi} and surjectivity of $\phi$ in \eqref{eq:map-phi}.
	
	 Note that for all $i\in [n]$, $S_i\hookrightarrow \psi(\N(X))$ is a smooth embedding and hence, a submanifold of $\psi(\N(X))$. To show \ref{def212:it2}, we need to show $\phi^{-1}(S_i)$ is a submanifold for all $i\in[n]$. Since, $\phi$ is transversal to $S_i$ for all $i\in [n]$, by the Preimage Theorem (\Cref{thm:preimage_thm}), $X_i=\phi^{-1}(S_i)$ is a submanifold of $X$.
	
	To show \ref{def212:it3}, note that by \Cref{thm:strat-sphere}, $\tilde{S}_i$ is locally closed in $\Rp^{n-1}$. Therefore, each $\tilde{S}_i$ can be written as the set difference of two open sets $U_i$ and $V_i$ in $\Rp^{n-1}$ for all $i=0,\ldots, n-1.$
	Consequently, 
    $$S_i=\tilde{S}_i\cap \psi(\N(X))=(U_i\setminus V_i)\cap \psi(\N(X)).$$ 
    Since $S_i=(U_i\cap \psi(\N(X)))\setminus (V_i\cap \psi(\N(X)))$ is a set difference of two open sets in $ \psi(\N(X))$, $S_i$ is locally closed in $\psi(\N(X))$. Moreover, the preimage of a locally closed set under a continuous map is a locally closed set. Hence, $X_i$ is locally closed for all $i=0,\ldots,n-1$ by \Cref{lem:composedmap}.

	Criterion \ref{def212:it4} holds by construction. Indeed, note that $\phi(X_i)= S_i$ and if $i\neq j$, then \Cref{thm:strat-sphere} gives that $S_i\cap S_j=\emptyset$. Hence, $X_i\cap X_j=\emptyset.$
	
	Finally, we prove \ref{def212:it5}. Since $X$ is smooth and connected, $\pn$ is also smooth and connected. Note that for $i>j$, we have that $S_i\cap \overline{S_j}=\emptyset$. Indeed,  $S_k=\tilde{S_k}\cap \psi(\N(X))$ and $\overline{S_k}\subseteq\overline{\tilde{S_k}}\cap \psi(\N(X))$ for all $k$, together with $\tilde{S_i}\cap \overline{\tilde{S_j}}=\emptyset$ for $i>j$ (by \Cref{thm:strat-sphere}) implies $S_i\cap \overline{S_j}=\emptyset$ for $i>j$. 
    
    Let $x$ be a point in $S_i$ for some $i$ and $S_i\cap \overline{S_j}\neq \emptyset$ for some $j>i.$ Then $x \in \varphi(C_F)$ for some $F\in \F_i$. Assume to the contrary that $x \notin \overline{S_j}$. This implies that $x\notin \overline{\varphi(C_F)\cap \pn}$ for all $F\in \F_j$. Since $\pn$ is connected, then either there exists an open neighborhood $U$ around $x$ such that $\psi(\N(U))$ is not full dimensional or $x$ is a boundary point.  Since $\phi$ is transversal and $X$ contains no boundary points,
    this gives a contradiction.
\end{proof}

The stratification constructed in the above theorem distinguishes the points on $X$ based on the expected dimension of the Voronoi cell of each point. In particular, the Voronoi cones of the points in the stratum $X_i$ have dimension $n-i$ and hence, their corresponding Voronoi cells have dimension at most $n-i$ by \Cref{rem:conedim}.

\begin{remark}
	Note that, for $F\in \F_i$, $C_F$ has dimension $i.$ Hence, so does $\tilde{S}_i$ in the stratification $\tilde{\S}_B$ of $\Rp^{n-1}$. Furthermore, since $\phi$ is an open and continuous map, generically, the strata $X_i$ of $\mathcal{X}$ also has dimension at most $i.$
\end{remark}

\begin{remark} In \Cref{thm:stratgeneral}, we stratify a connected component of a real manifold without boundary. If $X=\sqcup_{i=1}^n M_i$ is a real manifold with finitely many connected components $M_i$ such that each $M_i$ satisfies the conditions of \Cref{thm:stratgeneral}, we define the stratification of $X$ to be the union of stratification of each of the connected component $M_i.$
\end{remark}

\begin{example}\label{ex:torus}
     \begin{figure}[h!]
        \centering
        \includegraphics[width=0.25\linewidth]{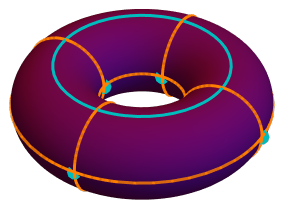}
        \caption{Partition of torus $X$ when the unit ball is the cube as in \Cref{ex:torus}. $X_0$ is the set of points in light blue, $X_1$ is the set of points in orange, and $X_2$ is the set of points in purple. }
        \label{fig:strat_torus}
    \end{figure}

    The transversality hypothesis in \Cref{thm:stratgeneral} is necessary to guarantee that $\mathcal{X}_B$ is a stratification of $X.$ As an example of the behavior that can arise when this hypothesis is not satisfied, consider when $B$ is the  cube, and take the manifold $X$ to be the torus defined by the equation $ (3 + x^2 + y^2 + z^2)^2-16 (x^2 + y^2)=0.$
    
    The strata are
    \begin{equation}
        \begin{split}
            X_0 = &\{(x,y,z)\in \RR^3 \mid z=1 \mbox{ and } x^2+y^2=4\} \cup \{(x,y,z)\in \RR^3 \mid z=-1 \mbox{ and } x^2+y^2=4\}\\
                 &\cup \{(1,0,0),(-1,0,0),(3,0,0),(-3,0,0),(0,1,0),(0,-1,0),(0,3,0),(0,-3-0)\} \\
            X_1 = & (\{(x,y,z)\in \RR^3 \mid z=0 \mbox{ and } x^2+y^2-1=0\} \cup \{(x,y,z)\in \RR^3 \mid z=0 \mbox{ and } x^2+y^2-9=0\} \\
            & \cup \{x=0 \mbox{ and } (y-2)^2 +z^2 -1=0\} \cup \{x=0 \mbox{ and } (y+2)^2 +z^2 -1=0\} \\
            & \cup \{y=0 \mbox{ and } (x-2)^2 +z^2 -1=0\} \cup \{y=0 \mbox{ and } (x+2)^2 +z^2 -1=0\} ) \setminus X_0\\
            X_2 = & X \setminus (X_1 \cup X_0).
        \end{split}
    \end{equation}
    
    These strata are depicted in \Cref{fig:strat_torus}, where $X_0$ is the union of two circles and eight points in blue, $X_1$ is the union of six circles in orange, and $X_2$ is depicted in purple. In particular, $X_0\not\subset \overline{X_1},$ this occurs precisely because the map $\phi$ is not transversal to $S_0$.

\end{example}

 \Cref{thm:stratgeneral} refers to smooth manifolds without boundary. Since smooth real algebraic varieties are manifolds without boundaries, the theorem applies and we use this to show next that each stratified part of codimension-one real algebraic varieties is in fact a semialgebraic set. 

\begin{theorem}
    Let $X=\V(f)$ be a pure codimension-one variety and let $\mathcal{X}_B:=\{X_i\}_{i=0}^{n-1}$ be its stratification. Then, for all $i = 0,\ldots, n-1$, $X_i$ is a semialgebraic set.
\end{theorem}

\begin{proof}
Given a pure codimension-one variety, $X=\V(f)$, to find the set of 
points in strata $X_i$, we note that these are exactly the points on $X$ whose normal vector is in $\bigsqcup_{F\in\F_i}C_F.$ 
Given a face $F$ of dimension $n-i-1$, let $w_1,\ldots, w_i$ be the generators of the normal cone of $F$. A normal vector $v\in C_F$ if and only if $v=\sum_{j=1}^i\lambda_jw_j$ for $\lambda_j>0.$
Thus, the set of points on $X$ that have $F$ as its type are given by the projection of the following set onto the first $n$ coordinates: 
\begin{align*}
\{(x_1,\ldots,x_n,\lambda_1,\ldots,\lambda_i)\in \RR^n\times \RR^i_{> 0} ~\mid~ (x_1,\ldots,x_n) \in X ,~ \{w_1,\ldots,w_i\}= {\rm gens}(C_F) \text{ and } \nabla_xf =\sum_{j=1}^i\lambda_jw_j \}.
\end{align*}
This set is semialgebraic and, therefore, by the Tarski-Seidenberg theorem, its projection is also semialgebraic. Using this we can now write the semialgebraic set describing $X_i$ as:

\begin{align*}\smash{\bigsqcup_{F\in \F_i}} \{(x_1,\ldots,x_n)\in \RR^n ~\mid~ &\exists (\lambda_1,\ldots,\lambda_i)\in \RR^i_{> 0} \text{ such that }(x_1,\ldots,x_n) \in X , \\ & \{w_1,\ldots,w_i\}= {\rm gens}(C_F) \text{ and } \nabla_xf -\sum_{j=1}^i\lambda_jw_j=0 \},\end{align*}
which is a finite union of semialgebraic sets.
\end{proof}

The above description can be used to explicitly compute the stratification of varieties as depicted in the following examples.

\begin{example}
	In \Cref{fig:plane_curves_stratified}, we present the stratification of the curves in \Cref{ex:type_plane_curves} with respect to the polyhedral norm induced by the square. The strata $(V_i)_0$ are depicted in blue, and the strata $(V_i)_1=V_i\setminus (V_i)_0$ are depicted in red for $i=1,2,3,4$. 
	
	\begin{figure}[h!]
		\centering
		\begin{subfigure}[t]{0.22\textwidth}
			\centering
			\includegraphics[scale=0.28]{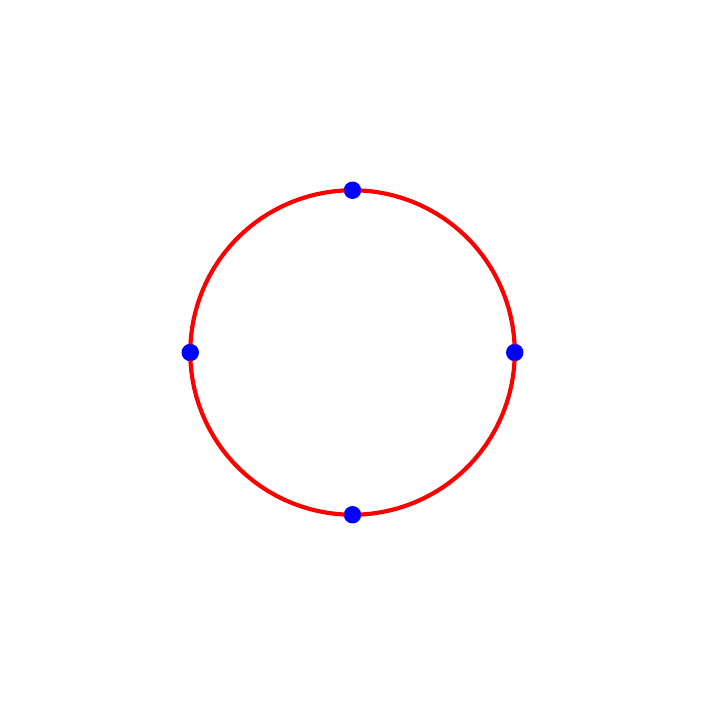}
			\caption{Stratification of $V_1$}
		\end{subfigure}
		\hfill
		\begin{subfigure}[t]{0.22\textwidth}
			\centering
			\includegraphics[scale=0.28]{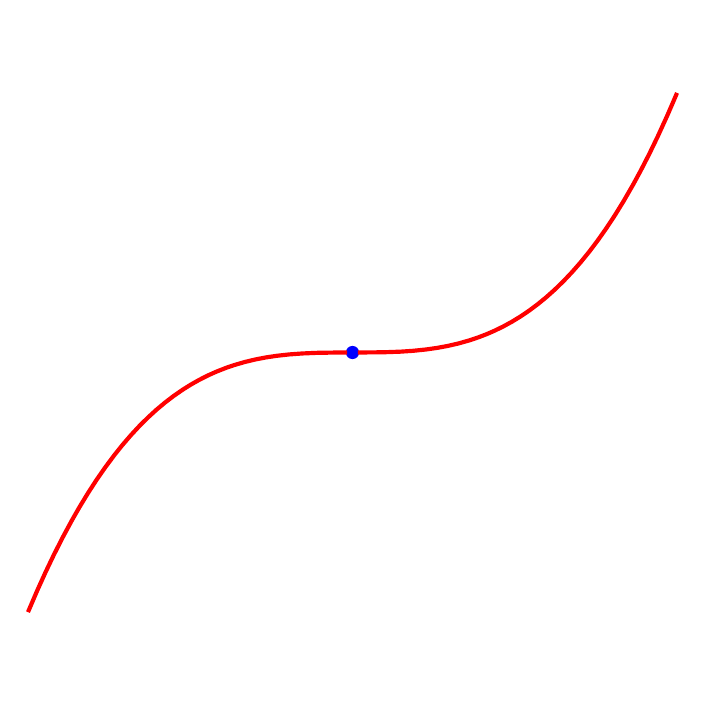}
			\caption{Stratification of $V_2$}
		\end{subfigure}
		\hfill
		\begin{subfigure}[t]{0.22\textwidth}
			\centering
			\includegraphics[scale=0.28]{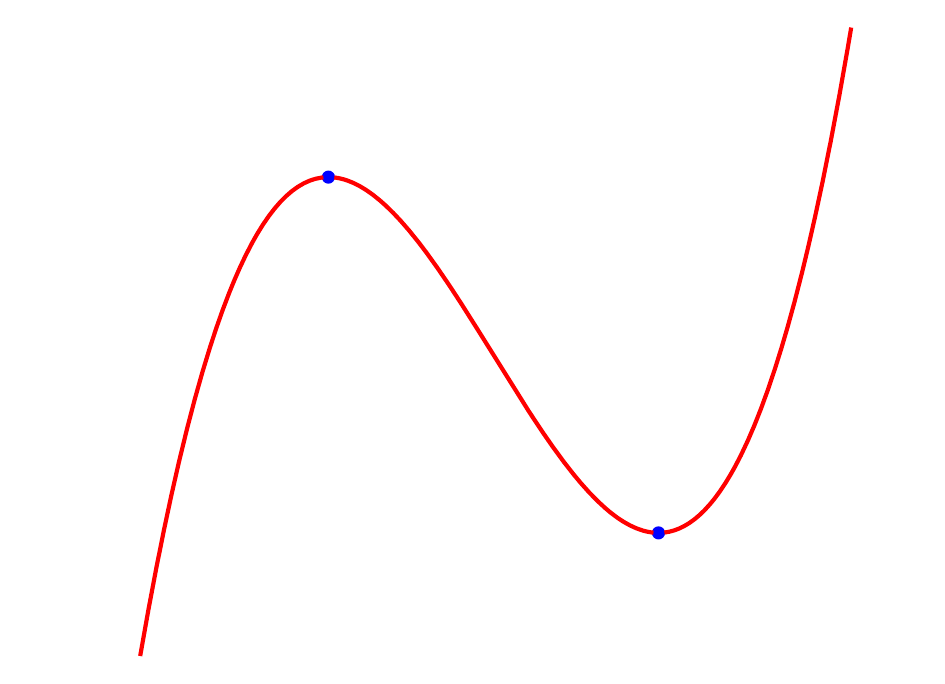}
			\caption{Stratification of $V_3$}
		\end{subfigure}
		\hfill
		\begin{subfigure}[t]{0.22\textwidth}
			\centering
			\includegraphics[scale=0.28]{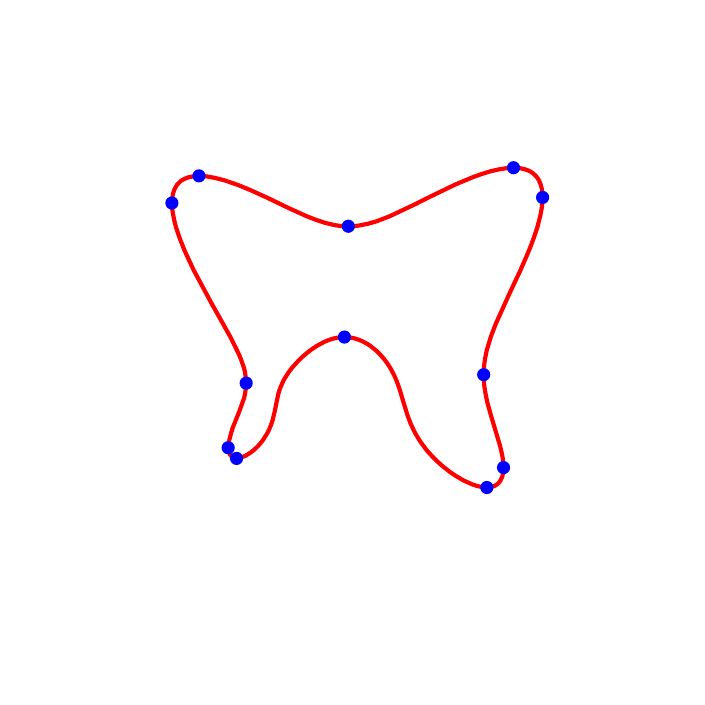}
			\caption{Stratification of $V_4$}
		\end{subfigure}
		\caption{Stratification of the curves in \Cref{ex:type_plane_curves}. For each curve $V_i$, $i=1,2,3,4$, the stratum $(V_i)_0$ is the set of points depicted in blue, and the stratum $(V_i)_1$ is the set of points depicted in red.}
		\label{fig:plane_curves_stratified}
	\end{figure}

\end{example}

\begin{example}
	Consider the hyperboloid $X$ from \Cref{ex:hyperbolic_example} with the norm induced by the cube. The stratification given by \Cref{thm:generalresult} is shown in \Cref{fig:hyperboloid_stratified}. The strata are:
	\begin{equation*}
		\begin{split}
			X_0 = &\{(1,0,0),(-1,0,0),(0,2,0),(0,-2,0)\}, \\
			X_1 = &\big(\{(x,y,z)\in\RR^3 \mid 4x^2+y^2=4 \mbox{ and } z=0 \} \cup \{(x,y,z)\in\RR^3 \mid 9x^2-z^2=9 \mbox{ and } y=0\}  \\ 
			&\ \ \cup \{(x,y,z)\in \RR^3 \mid 9y^2-4z^2=36 \mbox{ and } x=0 \} \big)\setminus X_0,\\
			X_2 = & X\setminus (X_0\cup X_1).
		\end{split}
	\end{equation*}
	These strata are depicted in \Cref{fig:hyperboloid_stratified} in orange, black, and blue, respectively. 
	
	\begin{figure}[h!]
		\centering
		\includegraphics[scale=0.2]{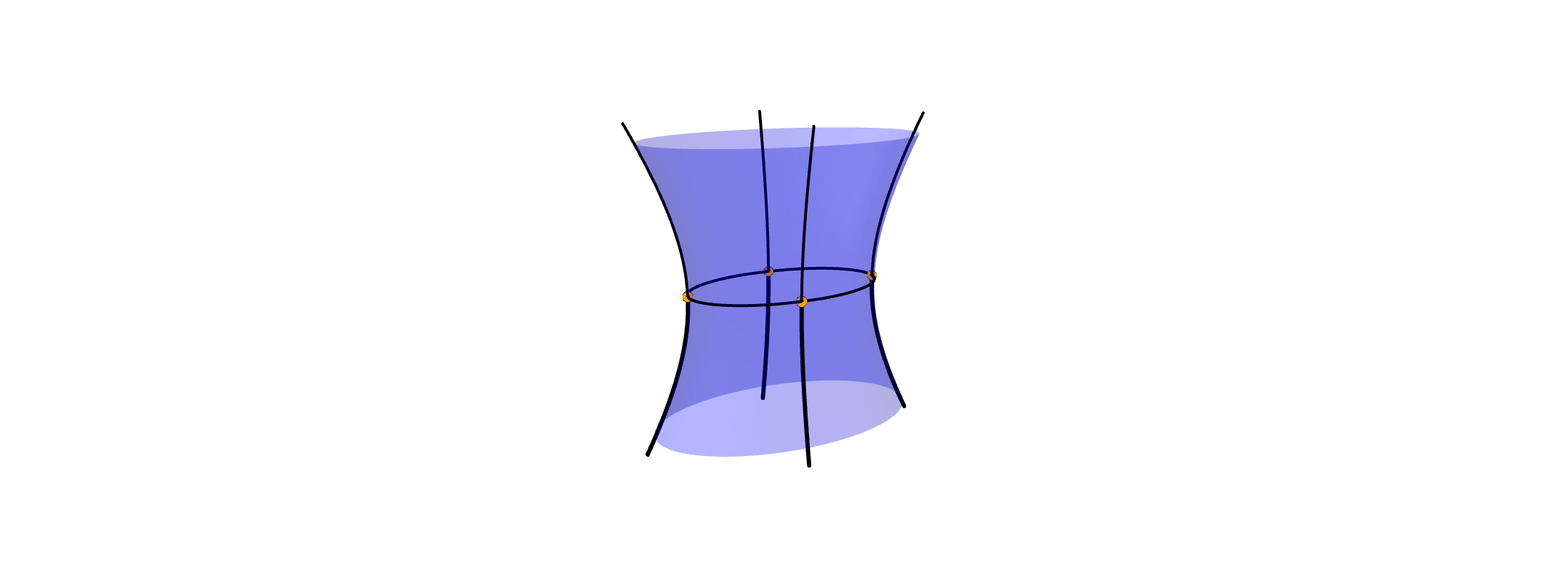}
		\caption{Stratification of the hyperboloid $X = \V(36x^2+9y^2-4z^2-36)$ induced by the polyhedral norm with the cube as its unit ball. The strata $X_0,X_1$, and $X_2$ are depicted in orange, black, and blue, respectively.}
		\label{fig:hyperboloid_stratified}
	\end{figure}
\end{example}

\begin{remark}
	It is possible that for a variety $X$, some of the strata may be empty. This is the case for hyperplanes in $\RR^n$. Moreover, \cite[Theorem 5.1]{BECEDAS2024102229} gives an upper bound on the number of full-dimensional Voronoi cells that a variety can have. This is equivalent to the upper bound on the number of points in the stratum $X_0$ in the stratification from \Cref{thm:stratgeneral}. 
\end{remark}

\section{Medial Axis}\label{sec:MA}

In \Cref{sec:Codim1}, we stratified a codimension-one variety $X \subseteq \mathbb{R}^n$ based on the dimension of the Voronoi cones of points on $X$. We now focus on finding the locus of points in $\RR^n$ whose distance to $X$ is minimized by at least two distinct points on $X$ with respect to a given polyhedral norm, $h$. Such a set is called the \textit{medial axis} and is denoted by $\med(X)$. 

\subsection{Finding the Medial Axis}

We begin by establishing a framework for the computation of an algebraic set that contains $\med(X)$. First, we formally define the medial axis.

\begin{definition} \label{def:medial} Given an algebraic variety $X \subseteq \mathbb{R}^n$ and a polyhedral norm $h$, the \emph{medial axis} of $X$ with respect to $h$ is
	\[\med(X):=\left\{u\in \RR^n ~~\mid~~~ \left|\arg\min_{x\in X} \ h(u-x)\right|> 1\right\}.\]
\end{definition}

\begin{example}\label{ex:med_axis_low_dim}
    For $X=\V(x^2-y)\subset \RR^2$ and the polyhedral norm $h$ with unit ball given by the square as described in \Cref{ex:plane_curve}, 
    \begin{equation*}
        \med(X)=\{(u_1,u_2)\in \RR^2 \mid u_1=0 \mbox{ and } u_2> 0\}.
    \end{equation*}
\end{example}

\begin{example} \label{ex:med_axis_everything_and_nothing}
Let $X = \V(x) \subset \RR^2$ be the line given by $x=0$ in $\RR^2$ and consider the polyhedral norm $h$ with unit ball given by the square as described in \Cref{ex:plane_curve}. Every point $u\in \RR^2 \setminus X$ is optimized by infinitely many points so the medial axis is a semialgebraic set of dimension 2 and is given by 
\[ \med(X) = \RR^2 \setminus X. \]
In this case, the variety $X$ is parallel to one of the facets of the unit ball.
This behavior is not possible in the Euclidean norm unless $X$ contains a circle. 

Alternatively, consider the variety $Y = \V(x + y) \subseteq \RR^2$
with the same polyhedral norm. In this case, no point in $\RR^2 \setminus Y$ is optimized by more than a single point on $Y$. Consequently, $\med(Y) = \emptyset$.
\end{example}

\begin{example}\label{ex:pos_dim_medial_axis}
	Let $X = \V(f)) \subset \RR^2$ where $f=y - 2x^2 + 2x^4$
	and $h$ be the polyhedral norm with unit ball given by a square as described in \Cref{ex:plane_curve}. Here $X$ is given by the blue curve in \Cref{fig:A2}. This curve has a real bitangent given by the line $y=1/2$ which is parallel to the optimizing facet corresponding to any point in the green cone. The two points of tangency, $\left(-\frac{1}{\sqrt{2}},\frac{1}{2}\right)$ and $\left(\frac{1}{\sqrt{2}},\frac{1}{2}\right)$, have two-dimensional Voronoi cells described by the dotted red cones. The medial axis of $X$ is the semialgebraic set of dimension two in green, given by 
	\begin{align*}\med(X)&=\left\{(x,y)~\mid~x - y +\frac{1+\sqrt{2}}{2}  < 0 \text{ and } x + y - \frac{1+\sqrt{2}}{2} > 0\}\cup \{(x,y)~\mid~x=0 \text{ and }y>0\right\}\\ &\cup \left\{(x,y)~\mid~x=0 \text{ and }y<-1\right\} \cup \{(x,y)~\mid~g(x,y)=0,~ f(x,y)<0,~x\leq -0.5 \text{ and }y\leq0.5\}\\ & \cup \{(x,y)~\mid~g(x,y)=0, ~f(x,y)<0, ~x\geq 0.5 \text{ and }y\leq 0.5\}\end{align*}
	where $g(x,y)=256 x^8 - 256 x^6+ 96 x^4- 12x^2+32x^2y- 64x^4 y+4 y^2-4y -1.$
	\begin{figure}[h!]
		\centering
		\includegraphics[scale=0.5]{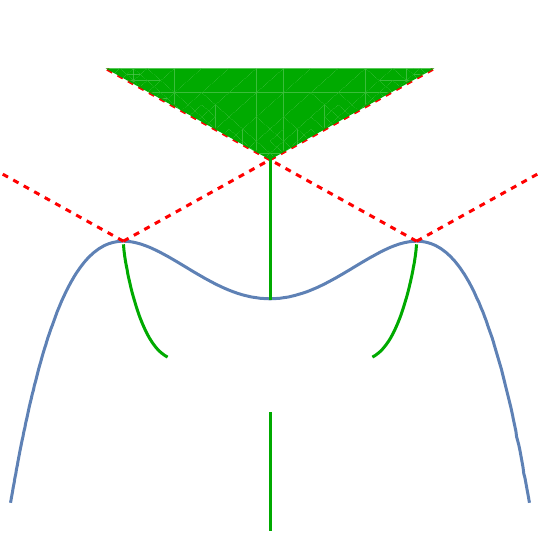}
		\caption{The variety $X$ from \Cref{ex:pos_dim_medial_axis} with full-dimensional medial axis.}
		\label{fig:A2}
	\end{figure}
\end{example}

\Cref{ex:med_axis_low_dim,ex:med_axis_everything_and_nothing,ex:pos_dim_medial_axis} show that the medial axis of a variety $X \subset \RR^n$ can exhibit a range of different behaviors. In particular, it can be empty, full-dimensional, or contain components of different dimensions. 	
In what follows, we focus on finding an algebraic set that contains the medial axis of a codimension-one variety $X=\V(f)$ for some polyhedral norm $h$ whose unit ball is $B.$  

Let $u\in \RR^n\setminus X$ be a point on the medial axis. Then there exist distinct points $x,y \in X$ such that  $h(x-u)=h(y-u)=\lambda$ for $\lambda>0.$ Let $B_\lambda(u)$ intersect $X$ at $x$ and $y$ in faces $F_1\in\F_i$ and $F_2\in\F_j$ of codimension $i$ and $j$ respectively. That is, $F_1$ and $F_2$ are the respective optimizing faces of $x$ and $y$. Assume that $\ell_1, \ldots , \ell_i$ are $i$ linear functionals defining $F_1$ and $\hat{\ell_1}, \ldots , \hat{\ell_j} $ are $j$ linear functionals defining $F_2$.

We define $J_{F_1}\subset \RR[x_1,\ldots,x_n,y_1,\ldots,y_n,u_1,\ldots,u_n,\lambda]$ to be the ideal generated by 

\begin{equation*}
   ( f(x), \ \ell_1(x-u)-\lambda, \ \ldots , \ \ell_i(x-u)-\lambda )
\end{equation*}
and the  $(i+1) \times (i+1)$ minors of the following matrix of dimension $(i+1) \times n$:
\begin{equation*} 
	\begin{bmatrix}
		\nabla_x f &
		\nabla_x \ell_1 &
		\cdots &
		\nabla_x \ell_i
	\end{bmatrix}^{\top}.
\end{equation*}

The above conditions impose that $x\in X$, $u\in\vc(x)$ , $h(x-u)=\l$  and the distance $\lambda$ is realized in the face $F_1$. 
Note that if $F_1$ is a vertex (i.e $i=n$), there are no $(n+1) \times (n+1)$ minors in the matrix defined above, so $J_{F_1}$ does not contain any determinantal equations.
Similarly, define $J_{F_2}$ as the ideal generated by 
\begin{equation*}
   ( f(y), \  \hat{\ell}_1(y - u) - \lambda , \  \ldots, \ \hat{\ell}_j(y - u) - \lambda )
\end{equation*} 
as well as the $(j + 1) \times (j + 1)$ minors of the $(j+1) \times n$ matrix 
\begin{equation*} 
	\begin{bmatrix}
		\nabla_y f &
		\nabla_y \hat{\ell}_1 &
		\cdots &
		\nabla_y \hat{\ell}_j
	\end{bmatrix}^{\top}.
\end{equation*}

Define the ideal 
\begin{equation}\label{eq:incidenceideal}
	J_{(F_1,F_2)}:=J_{F_1}+J_{F_2}.
\end{equation}
and denote by $I_{(F_1,F_2)}\subset \RR[u_1,\ldots,u_n]$ the elimination ideal $J_{(F_1,F_2)}\cap \RR[u_1,\ldots,u_n].$ By construction, $\V(I_{(F_1,F_2)})$ contains the points in $\RR^n$ that are equidistant from two points in $X$ and whose distance is realized in faces $F_1$ and $F_2$. This is precisely the projection of the incidence variety

\begin{align}\begin{split}\label{eq:IV}
	\mathcal{V}_{(F_1,F_2)}:=\{(x,y,u,\lambda)\in \RR^n\times\RR^n\times \RR^n \times \RR~\mid~   &x,y\in X, \  h(x-u)=h(y-u)=\lambda, \\ &F_\l(u,x)=F_1,\ F_\l(u,y)=F_2 \}
\end{split}\end{align}
\noindent onto the coordinates $u_1,\ldots, u_n$ (where we recall that $F_\l(u,v)$ is the face of $B_\lambda (u)$ intersecting $X$ at $v$). 

Observe that taking the union of the varieties $\V(I_{(F_1,F_2)})$ over all faces $F_1$ and $F_2$ yields the set of points that are equidistant to $X$ and satisfy the necessary (though not sufficient) tangency conditions for optimality. We call 
\[ I_{\text{Eq}} = \bigcap_{G\in \F \times F}I_{G} \]
the \emph{equidistant ideal} and $\V(I_{\text{Eq}})$ the \emph{equidistant locus} of $X$.
The equidistant locus of $X$ may be strictly larger than the medial axis of $X$ and its algebraic closure due to the fact that we do not impose minimality conditions on $\lambda$ when generating the ideals $J_{F_1}$ and $J_{F_2}$. We now show that the equidistant locus gives an algebraic set containing the medial axis of $X$.

\begin{theorem}\label{thm:medialaxis}
	Let $X \subseteq \RR^n$ be a smooth codimension-one variety, $h$ be a polyhedral norm whose unit ball has 
    faces $\F$ and 
    \[ I_{\mathrm{Eq}} = \bigcap_{G\in \F \times F}I_{G}\]
    be the equidistant ideal of $X$. Let $M_X$ be the algebraic closure of $\med(X)$. Then
	\[\med(X) \subset M_X\subset \mathbb{V} \left(I_{\mathrm{Eq}} \right).\]
\end{theorem}

\begin{proof}
 Note that $\med(X) \subset M_X$ follows from the definition of algebraic closure. 
To show that $M_X \subset \mathbb{V}\left(I_{\mathrm{Eq}}\right)$, we will show that $\med(X) \subset \mathbb{V}\left(I_{\mathrm{Eq}}\right)$ from the definitions of these sets. 
 Since $\mathbb{V}\left(I_{\mathrm{Eq}}\right)$ is closed in the Zariski topology, this will imply that 
$M_X\subset \mathbb{V}\left(I_{\mathrm{Eq}}\right) $. 

To show that $\med(X) \subset \mathbb{V}\left(I_{\mathrm{Eq}}\right)$, let $u \in \med(X)$. Then there exist distinct $x,y \in X$ such that  $ x,y \in \arg\min\limits_{x\in X} \ h(u-x)$. Let $\lambda= h(u-x)=h(u-y)$. Let $F_{i}$ and $F_{j}$ be the faces at which $B_{\lambda}(u)$ intersects $x$ and $y$, respectively. 
Then $(x, y, u, \lambda) \in \mathcal{V}_{(F_1,F_2)}$, so $u \in \mathbb{V} (I_{(F_i, F_i)})$. For any ideals $I$ and $J$, we have $\mathbb{V}( I \cap J)= \V(I) \cup \V(J)$. Thus, $u \in \mathbb{V}\left(I_{\mathrm{Eq}}\right)$, completing the proof. 
\end{proof}

\begin{remark}
	When $X$ is a semialgebraic set, $\med(X)$ is also semialgebraic. The argument is similar to the Euclidean norm case \cite[Proposition 7.1]{breiding2024metric}.
\end{remark}

\begin{remark}\label{rem:extracomponents}
As mentioned earlier, due to the fact that we do not impose minimality conditions in our definition of $I_{\mathrm{Eq}}$, the equidistant locus $ \mathbb{V}\left( I_{\mathrm{Eq}} \right)$  may be strictly larger than the algebraic closure of $\med(X)$. 
In particular, $\V(I_{\mathrm{Eq}})$ may contain additional components. This is illustrated in \Cref{ex: extra components}.
\end{remark}

\begin{example}\label{ex: extra components}
	\begin{figure}[h!]
    \centering
	\begin{subfigure}[t]{0.22\textwidth}
		\centering
		\includegraphics[scale=0.15]{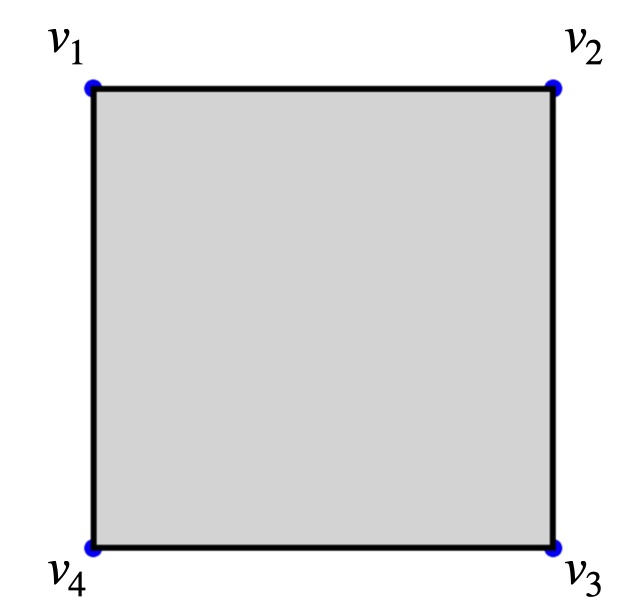}
		\caption{Polyhedral ball with labeled vertices}
        \label{fig:circ_labeled_ball}
	\end{subfigure}
    \hfill
    \begin{subfigure}[t]{0.22\textwidth}
        \centering
        \includegraphics[scale=0.28]{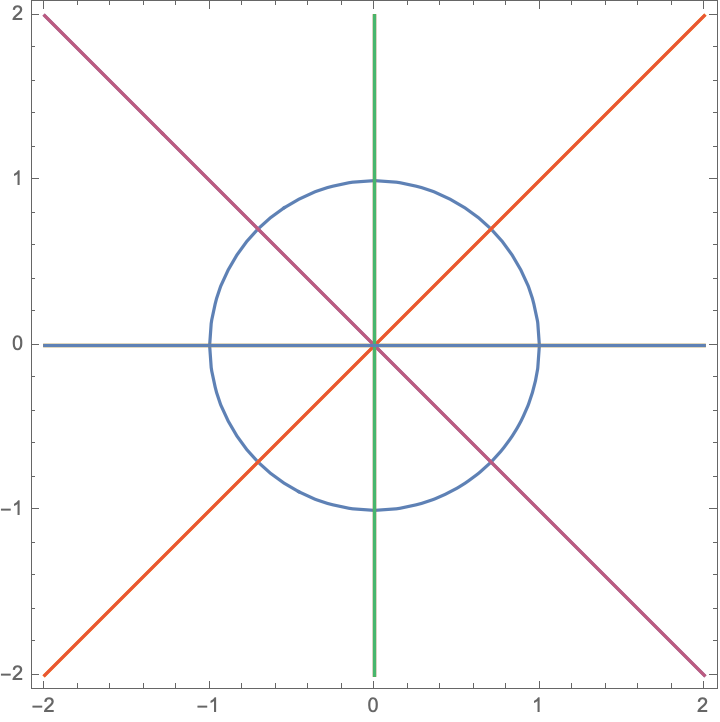}
        \caption{$\V(I_{F_1,F_2})$ when $F_1$ and $F_2$ are vertices}
        \label{fig:circ_v_v}
    \end{subfigure}
    \hfill
    \begin{subfigure}[t]{0.22\textwidth}
        \centering
        \includegraphics[scale=0.28]{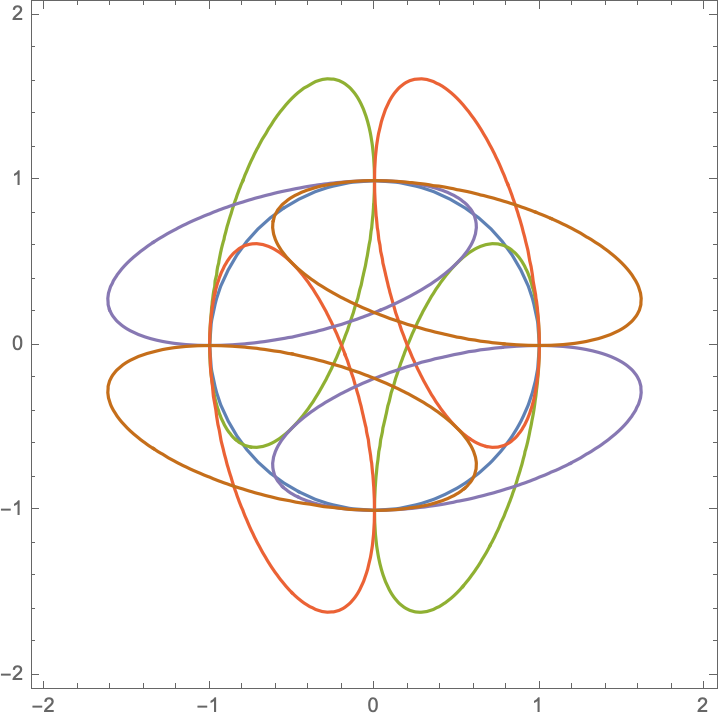}
        \caption{$\V(I_{F_1,F_2})$ when one face is a vertex and one is an edge}
        \label{fig:circ_v_e}
    \end{subfigure}
    \hfill
    \begin{subfigure}[t]{0.22\textwidth}
        \centering
        \includegraphics[scale=0.28]{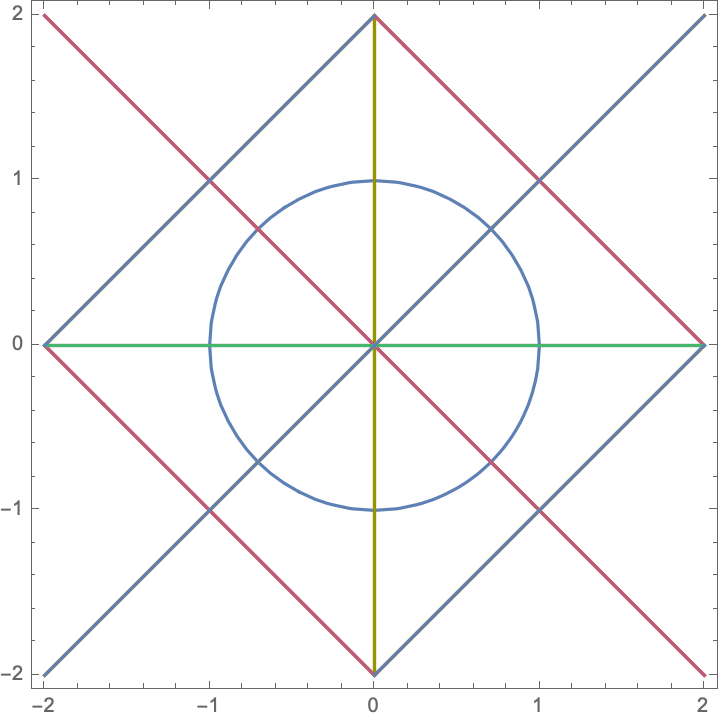}
        \caption{$\V(I_{F_1,F_2})$ when $F_1$ and $F_2$ are edges}
        \label{fig:circ_e_e}
    \end{subfigure}
    \hfill
		\caption{Consider the variety $X = \V(x^2 + y^2 -1) \subset \RR^2$ with respect to polyhedral norm with unit ball given by the square. Each figure represents the algebraic closure of the locus of equidistant points arising from different pairs of faces. The variety $X$ is plotted in all the pictures for reference.
        }
		\label{fig:med_axis_circ}
	\end{figure}

We illustrate the discussion in \Cref{rem:extracomponents} for the case when $X = \V(x^2 + y^2 -1) \subset \RR^2$ and the polyhedral norm has unit ball given by the square as described in \Cref{ex:plane_curve}. The nontrivial faces of a square in $\RR^2$ are either zero or one dimensional.
Therefore, the components of the medial axis are generated with respect to two vertices, a vertex and an edge, or two edges. 

In \Cref{fig:circ_v_v}, we plot the zero set of varieties of $I_{(F_1,F_2)}$ where the pair $(F_1, F_2)$ is a pair of vertices. The green line is the Zariski closure of the component corresponding to face pairs $(v_1,v_2)$ and $(v_3,v_4)$ where the vertices are labeled in \Cref{fig:circ_labeled_ball}. Similarly, the blue line is the zero set of $I_{(v_2,v_3)}$ and $I_{(v_1,v_4)}$. These two lines are in fact the Zariski closures of components of medial axis. 

On the other hand, the orange line, given by $\mathbb{V}(I_{(v_1,v_3)})$, is not a component of the medial axis. For example, the point $u=\left(\frac{1}{2\sqrt{2}},\frac{1}{2\sqrt{2}}\right)$ lies on the orange line and has distance $d(u,X)=1/2$. However, $u$ is equidistant from points $\left(\frac{\sqrt{2}+\sqrt{6}}{4}, \frac{\sqrt{2}-\sqrt{6}}{4}\right)$ and $\left(\frac{\sqrt{2}-\sqrt{6}}{4}, \frac{\sqrt{2}+\sqrt{6}}{4}\right)$ with distance $\frac{\sqrt{3}}{2}>d(u,X)$. 

Similarly, the ellipses in \Cref{fig:circ_v_e} are generated by considering pairs of faces $(F_1,F_2)$ where $F_1$ is a vertex and $F_2$ is an edge. 
Finally, \Cref{fig:circ_e_e} is generated by considering $I_{(F_1,F_2)}$ where $F_1$ and $F_2$ are both edges of the square.
\end{example}

Note that if $\med(X)$ contains a full-dimensional component corresponding to given faces $F_1$ and $F_2$, the ideal $I_{(F_1,F_2)}$ is the zero ideal and its algebraic closure is $\RR^n$. This behavior does not arise very often and we believe it is related to the existence of families of affine spaces that are bitangent to $X$. We state a conjecture for a necessary condition for the existence of a full-dimensional component of the medial axis. 

\begin{conjecture}\label{conj:fulldimcomp}
	Let $X$ be a codimension-one variety and $\X_B=\cup_{i}X_i$ be its stratification with respect to a given polyhedral norm (assuming this stratification exists). If the medial axis has a full-dimensional component, then for some $i$ there exist two open sets $M_{i1},M_{i2}\subset X_i$ such that every point in $M_{i1}$ shares a bitangent affine space with a point in $M_{i2}$ and these spaces are contained in supporting hyperplanes of one of the faces of $B$.
\end{conjecture}

\subsection{Degree Bounds}

Let $X\subset \RR^n$ be a codimension-one variety given by $f(x)=0$ such that $\deg(f)=d$. 
In this setting, we compute degree bounds of $\IF$ and hence, degree bounds on various components of the equidistant locus $\V(I_{\mathrm{Eq}}).$ Throughout the remainder of this section, we assume that the medial axis does not have a full-dimensional component. 
In the case that it does have a full-dimensional component, the ideal is the zero ideal and all of the degree bounds below are trivially satisfied.
We first consider the degree of $\IF$ in the case when $F_1$ and $F_2$ are both vertices of the given polyhedral unit ball. Before we state our main result, we need a lemma bounding the degree of resultants.

\begin{lemma}\label{lem: degree sylvester mat}
Let $K$ be a field and $f(t):=\sum_{i=1}^da_it^i$ and $g(t):=\sum_{i=1}^db_it^i$ be two univariate polynomials
	such that $a_i, b_i \in K[u_1,\ldots,u_n]$ and $\deg(a_i)=\deg(b_i)=d-i$. Then the resultant, ${\res}(f,g) \in K[u_1,\ldots,u_n]$ has degree at most $d^2$.
\end{lemma}

\begin{proof}
	Let $T$ be the Sylvester matrix of $f$ and $g$ of size $2d$. Its nonzero entries are:
	\begin{align*}
		T_{i,i+j} &= a_{d-j} ~~ \text{ for } j=0,\dots ,d  \mbox{ and } i=1,\dots, d \\
		T_{i,i-d+j} &= b_{d-j} ~~ \text{ for } j=0,\dots ,d  \mbox{ and } i=d+1,\dots, 2d.
	\end{align*} 
	This implies that 
	\begin{equation} \label{eq:deg_Sylvester}
		\left\{
		\begin{split}
			&\deg_u(T_{i,i+j})=j &\mbox{ for } j=0,\dots ,d  \mbox{ and } i=1,\dots, d \\
			&\deg_u(T_{i,i-d+j}) =j &\mbox{ for } j=0,\dots ,d  \mbox{ and } i=d+1,\dots, 2d.
		\end{split} \right.
	\end{equation}
	Let $\mathcal{S}_{2d}$ be the permutations of $[2d]$. Then ${\res}(f,g)$ is given by
	\[\res(f,g)=\sum_{\sigma \in \mathcal{S}_{2d}} \text{sign}(\sigma) T_{1,\sigma(1)}\cdots T_{d,\sigma(d)}T_{d+1,\sigma(d+1)}\cdots T_{2d,\sigma(2d)}.\]

	Let $\sigma \in \mathcal{S}_{2d}$ be such that any of the conditions below hold for some $i$ 
	
	\[\left\{ 
	\begin{array}{l}
		1\leq i \leq d \quad \mbox{ and } \quad \sigma(i) > i+d \\
		1\leq i \leq d \quad \mbox{ and } \quad \sigma(i) < i \\
		d+1 \leq i \leq 2d \quad \mbox{ and } \quad \sigma(i) > i \\
		d+1 \leq i \leq 2d \quad \mbox{ and } \quad \sigma(i) < i-d
	\end{array}
	\right.\] 
	then we have $T_{1,\sigma(1)}\dots T_{2d,\sigma(2d)}=0$.
	
	We now consider $\sigma\in\mathcal{S}_{2d}$ that does not satisfy any of the conditions above. In other words, we have $i\leq \sigma(i) \leq i+d$ for $i=1,\dots ,d$ and $i-d\leq \sigma(i) \leq i$ for $d+1\leq i \leq 2d$. By \eqref{eq:deg_Sylvester}, we have 
{\allowdisplaybreaks	\begin{align*}
		\begin{split}
			\deg_u(T_{1,\sigma(1)}\dots T_{2d,\sigma(2d)})= & \sum_{i=1}^{d}\deg_u(T_{i,\sigma(i)}) + \sum_{i=d+1}^{2d}\deg_u(T_{i,\sigma(i)})\\
			=& \sum_{i=1}^d (\sigma(i)-i) + \sum_{i=d+1}^{2d} (\sigma(i)-i+d) \\
			=& \sum_{i=1}^{2d} \sigma(i) -2 \sum_{i=1}^d i  \\
			=& \frac{2d(2d+1)}{2} - 2\frac{d(d+1)}{2}=d^2.
		\end{split}
	\end{align*}
}
	This shows that each nonzero term of the determinant of the Sylvester matrix has degree $d^2$. Hence, $\res(f,g)$ is the sum of polynomials of degree $d^2$, and its degree is at most $d^2$. 
\end{proof}

\begin{theorem}\label{thm:vertex_vertex}
	Let $X=\V(f)\subset \RR^n$ be a codimension-one variety and $\deg(f) = d$. Let $F_1$ and $F_2$ be two vertices of the polyhedral unit ball. Then $\IF=\langle f(u)g(u) \rangle$ for some $g\in \RR[u]$ and $\deg(\IF) \leq d^2$. 
    In particular, the component of $\med(X)$ arising from a pair of vertices has degree at most $d^2 - d$.
\end{theorem}

\begin{proof}
	Given vertices $F_1$ and $F_2$ of the unit ball, the ideal $\JF$ as in \eqref{eq:incidenceideal} is given by
	\begin{align}\label{eq:JFsysverts}
		\JF = \langle f(x), f(y), \ell_1(x-u) - \lambda,\ldots, \ell_n(x - u) - \lambda, \hat{\ell}_1(y-u) - \lambda,\ldots, \hat{\ell}_n(y - u) - \lambda \rangle 
	\end{align}
	where $\ell_1,\ldots,\ell_n$ (resp. $\hat{\ell}_1,\ldots,\hat{\ell}_n$) correspond to the linear functionals defining the face $F_i$ (resp. $F_j$). 
	Observe that this ideal is defined by $2n$ linear polynomials and two polynomials of degree $d$. We use Gaussian elimination to solve for $x$ and $y$ using the linear equations. Since the $\ell_i$ and $\hat{\ell}_i$ are linear functionals, they can be expressed in terms of the inner product as 
	\[\ell_i(x)=\langle l_i,x \rangle, \qquad \hat{\ell}_i(y)=\langle \hat{l}_i,y \rangle  \]
	such that $l_i, \hat{l}_i \in \Rd$ for all $i$.
	Define the $n\times n$ matrices $A_1$ and $A_2$ that have $\{l_1,\dots ,l_n\}$  and $ \{ \hat{l}_1,\dots,\hat{l}_n\} $ as rows respectively. Note that these matrices are invertible, since the vectors $\{l_1,\dots ,l_n\}$ (resp. $ \{ \hat{l}_1,\dots,\hat{l}_n\} $) are linearly independent by construction. We write the system of linear equations in \eqref{eq:JFsysverts} as
	\[\begin{bmatrix}
		A_1 & 0 & -A_1 & - \mathbf{1} \\
		0 & A_2 & -A_2 & -\mathbf{1}
	\end{bmatrix} 
	\begin{bmatrix}
		x \\
		y \\
		u \\
		\lambda 
	\end{bmatrix} = 0
	\]
	where $\mathbf{1}$ denotes the column vector with entries equal to $1$. Solving for $x$ and $y$, we get:
	\[x=u+A_1^{-1}\mathbf{1}\lambda \qquad \mbox{and} \qquad y=u+A_2^{-1}\mathbf{1}\lambda. \]
	 Substituting this into $f(x)$ and $f(y)$, we get $J_{u,\l}=\JF \cap R[u,\l]=\langle g_1(u,\l),g_2(u,\l)\rangle$ where
	 \begin{equation*}
	 	\begin{split}
	 		g_1(u,\lambda)\coloneqq & f(u+A_1^{-1}\mathbf{1}\lambda) = a_d \lambda^d + a_{d-1}\lambda^{d-1} + \ldots + a_0,\\
	 		g_2(u,\lambda)\coloneqq & f(u+A_2^{-1}\mathbf{1}\lambda) = b_d \lambda^d + b_{d-1}\lambda^{d-1} + \ldots + b_0.
	 	\end{split}
	 \end{equation*}
	where $a_i$ and $b_i$ are degree $d-i$ polynomials in the variables $u_1,\ldots,u_n$. 
	Since $\IF=J_{u,\l}\cap \RR[u]$, using \Cref{lem: degree sylvester mat}, $\IF$ is an ideal of a polynomial of degree at most $d^2$.  Furthermore, since $a_0=g_1(u,0)=f(u)=g_2(u,0)=b_0$, the determinant of the Sylvester matrix of $g_1$ and $g_2$ computed using the first column shows that $f(u)$ is a factor of ${\rm res}(g_1,g_2)$. Hence, $\IF=\langle f(u)g(u)\rangle$ for some polynomial $g(u)\in \RR[u]$ such that $\deg(g)\leq d^2-d.$
    
    Finally, after saturating the ideal with $f(u)$, we get that the component of $\med(X)$ arising from a pair of vertices has degree at most $d^2 - d$.
\end{proof}

Next, we consider the degree of $\IF$ when $F_1$ is a vertex and $F_2$ is a facet of the polyhedral ball.

\begin{theorem}\label{thm:vertex-facet}
	Let $X=\V(f)\subset \RR^n$ and $\deg(f) = d$. Let $F_1$ be a vertex and $F_2$ be a facet of the polyhedral ball. Suppose that $X$ has $m$ points with tangent spaces parallel to $F_2$. Then $\IF=I_1\cap\ldots\cap I_m$ is the intersection of $m$ ideals and $\deg(I_k)\leq d$ for all $k=1,\ldots,m.$
\end{theorem}

\begin{proof}
Let $X$ be the given variety and $\ell_1, \dots , \ell_n$ be $n$ linear functionals describing $F_1$. Let $z_1,\ldots,z_m\in X$ be the fixed $m$ points with tangent spaces parallel to $F_2$ and let $\hat{\ell}$ be the functional corresponding to $F_2$. The ideal $\IF$ corresponds to the variety of points that are equidistant from two points with optimizing faces $F_1$ and $F_2$. Since $F_2$ belongs to the $\type$ of only $m$ points on $X$, $\IF$ is given by: 

\begin{equation}
    \IF=\bigcap_{k=1}^m (J_{\{F_1,z_k\}}\cap \RR[u])
\end{equation}
where $J_{\{F_1,z_k\}}\subset \RR[x,u,\l]$ is given by
\begin{align*}
	J_{\{F_1,z_k\}}:= \langle f(x),\ell_1(x-u)-\lambda, \ldots, \ell_n(x-u) - \lambda, \hat{\ell}(z_k-u)-\lambda \rangle.
\end{align*}

For an arbitrary $k$, let $I_k:=J_{\{F_1,z_k\}}\cap \RR[u]$. To prove the result, we need to show that $\deg(I_k)\leq d$ for all $k=1,\ldots, m.$ To that end, after solving $n$ linear equations in $J_{\{F_i,z_k\}}$ for $x$, we get,
\[J_{u,\l}^{(k)}:=J_{\{F_1,z_k\}}\cap \RR[u,\l]=\langle f(u+A^{-1}\mathbf{1}\lambda), \hat{\ell}(z_k-u)-\lambda \rangle\]
where $A$ is a matrix of size $n$ with rows given by the $l_1,\ldots,l_n$ for $l_i$ linearly independent, such that $\ell_i(x)=\langle l_i,x \rangle.$ 
Finally, $I_k=J_{u,\l}^{(k)}\cap \RR[u]$ is generated by the resultant of $f(u+A^{-1}\mathbf{1}\lambda)$ and $ \hat{\ell}(z_k-u)-\lambda$ with respect to $\l$. Since $\deg(f)=d$ and $\hat{\ell}$ is a linear functional, $\deg(I_k)\leq d$ for all $k=1,\ldots,m.$
\end{proof}

Next, we consider the degrees of $\IF$ with respect to two facets of the unit ball. 

\begin{theorem}\label{thm:facet-facet}
Let $X=\V(f)\subset \RR^n$ and $\deg(f) = d$. Let $F_1$ and $F_2$ be two distinct facets of the polyhedral ball such that $X$ has $s$ points with tangent spaces parallel to $F_1$ and $t$ points with tangent spaces parallel to $F_2$. Then $\deg(\IF)\leq st$.
\end{theorem}
\begin{proof}
	Let $F_1, F_2$ be defined by linear functionals $\ell_1$ and $\ell_2$ respectively. Then, 
	\begin{align*}
		\JF &= \langle f(x), f(y), \ell_1(x - u) - \lambda, \ell_2(y-u) - \lambda, \ \text{rank}(M_1(x)) = 1, \ \text{rank}(M_2(y)) = 1 \rangle
	\end{align*}
	where $M_1(x)$ and $M_2(y)$ are defined as
	\begin{align*}
		M_1(x) = \begin{bmatrix}
			(\nabla f)(x) \\
			(\nabla \ell_1)(x)
		\end{bmatrix}, \qquad M_2(y) = \begin{bmatrix}
			(\nabla f)(y) \\
			(\nabla \ell_2)(y)
		\end{bmatrix}.
	\end{align*}
	
	Denote the set of points in $X$ tangent to a translate of $F_1$ as $p_1,\ldots,p_s$ and those tangent to a translate of $F_2$ as $q_1,\ldots,q_t$. Observe that for fixed $p \in \{p_1,\ldots, p_s\}$, $f(p) = 0$ and $\text{rank}(M_1(p)) = 1$ and similarly for fixed $q \in \{q_1,\ldots, q_t\}$, $f(q) = 0$ and $\text{rank}(M_2(q)) = 1$. Therefore, for such fixed $p,q$, the ideal $I$ becomes
	\[ J_{p,q} = \langle \ell_1(p-u) - \lambda, \ell_2(q - u) - \lambda  \rangle.  \]
	
	We can then project this ideal onto the $u$ coordinates by eliminating $\lambda$. We get
	\[I_{p,q}:=J_{p,q}\cap \RR[u]= \langle \ell_1(p-u) - \ell_2(q-u) \rangle. \]
	Hence, $\IF=\bigcap_{p,q}I_{p,q}$ where $p \in \{p_1,\ldots, p_s\}$ and $q \in \{q_1,\ldots,q_t\}$. Since $I_{p,q}$ has degree one and $\V(I_{p,q})$ is a hyperplane, we get $\deg(\IF)\leq st$ and  $\V(\IF)$ consists of at most $st$ hyperplanes.
\end{proof}

\begin{remark}[$n=2$]
	For codimension-one varieties in $\RR^2$, \Cref{thm:vertex_vertex}, \Cref{thm:vertex-facet}, and \Cref{thm:facet-facet} cover all possible cases and we have bounds for all the components of the equidistant locus.
\end{remark}

Finally, we conclude this section by giving a degree bound for quadratic varieties.

\begin{proposition}
	Let $X=\V(f)$ be such that $\deg(f) = 2$. Then $\deg(\IF)\leq 4$ for any two faces $F_1$ and $F_2$ of the unit ball.
\end{proposition}
\begin{proof}
	Let $F_i$ (resp. $F_j$) be defined by linear functionals $\ell_1,\ldots, \ell_i$ (resp. $\hat{\ell}_i,\ldots, \hat{\ell}_j$). Then  
	{\small \[ \JF = \langle f(x), f(y), \ell_1(x - u) - \lambda, \ldots, \ell_i(x-u) - \lambda, \hat{\ell}_1(y-u) - \lambda, \ldots, \hat{\ell}_j(y-u) - \lambda, \text{rank}(M_1(x)) = i, \text{rank}(M_2(y)) = j  \rangle \]} 
	where 
	\begin{align*} M_1(x) &= \begin{bmatrix}
		(\nabla f)(x) &
		(\nabla \ell_1)(x) &
		\cdots &
		(\nabla \ell_i)(x)
	\end{bmatrix}^\top, \\ M_2(y) &= \begin{bmatrix}
		(\nabla f)(y) &
		(\nabla \hat{\ell}_1)(y) &
		\cdots &
		(\nabla \hat{\ell}_j)(y)
	\end{bmatrix}^\top.\end{align*}
	Observe that since $\deg(f) = 2$, $\nabla f$ is linear in $x$ (or $y$) and hence, the rank constraints are linear. In particular, there are $n-i$ linearly independent equations in $x$ given by $\text{rank}(M_1(x)) = i$ and $n-j$ linearly independent equations in $y$ given by $\text{rank}(M_2(y)) = j$. Using these linear equations along with $\ell_m(x-u) - \lambda, \hat{\ell}_k(y-u) - \lambda $ for $m \in [i]$, $k \in [j]$ gives $2n$  linearly independent equations. By performing Gaussian elimination, we get $x_i = a_i(u,\lambda)$ and $y_i = b_i(u,\lambda)$ where $a_i$ and $b_i$ are linear functions of $u,\lambda$.  Consequently,
	\[ J_{u,\lambda} := \JF\cap \RR[u,\l]=\langle \tilde{f}_1(u,\lambda), \tilde{f}_2(u,\lambda) \rangle \]
	where $\tilde{f}_1(u,\lambda) = f(a_1(u,\lambda),\ldots,a_n(u,\lambda))$ and $\tilde{f}_2 = (b_1(u,\lambda),\ldots,b_n(u,\lambda))$. 
	Finally, $\IF= J_{u,\lambda}\cap \RR[u]$ and by \Cref{lem: degree sylvester mat}, we get $\deg(\IF)=2^2 = 4$.
\end{proof}

\subsection*{Acknowledgements}
{\small We thank Kathlén Kohn and Lorenzo Venturello for discussions during the early phases of the project. NK has been funded by the European Union under the Grant Agreement no. 101044561, POSALG. Views and opinions expressed are those of the authors only and do not necessarily reflect those of the European Union or European Research Council (ERC). Neither the European Union nor ERC can be held responsible for them. The author Eliana Duarte was partially supported by Centro de Matemática Universidade do Porto, member of LASI, which is financed by national funds through FCT – Fundação para a Ciência e a Tecnologia, I.P., under the projects with reference UID/00144/2025 and associated DOI given by https://doi.org/10.54499/UID/00144/2025.}

\bibliography{dist-opt}

\vspace{1cm}

\noindent
\footnotesize {\bf Authors' addresses:}

\smallskip

\noindent Eliana Duarte, Universidade do Porto
 \hfill  {\tt eliana.gelvez@fc.up.pt}

\noindent Nidhi Kaihnsa, University of Copenhagen
 \hfill  {\tt nidhi@math.ku.dk}

\noindent Julia Lindberg, Georgia Institute of Technology
 \hfill  {\tt jlindberg3@math.gatech.edu}

\noindent Angélica Torres, Universidad Técnica Federico Santa María
 \hfill  {\tt angelica.torresb@usm.cl}

\noindent Madeleine Weinstein, University of Puget Sound
 \hfill  {\tt mweinstein@pugetsound.edu}

\end{document}